\documentclass[12pt]{amsart}
\usepackage{amsmath, amssymb, amsfonts, amsthm, url}
\usepackage{mathrsfs}
\usepackage{geometry}
\usepackage[usenames, dvipsnames]{color}

\input xy
\xyoption{all}

\newcommand{\g}{\mathfrak{g}}

\newcommand{\n}{\mathfrak{n}}
\renewcommand{\b}{\mathfrak{b}}
\renewcommand{\t}{\mathfrak{t}}
\renewcommand{\u}{\mathfrak{u}}
\newcommand{\C}{\mathbb{C}}
\newcommand{\Z}{\mathbb{Z}}
\newcommand{\Q}{\mathbb{Q}}
\newcommand{\W}{\mathbb{W}}

\newcommand{\LL}{\mathcal{L}}
\newcommand{\VV}{\mathcal{V}}
\newcommand{\MM}{\mathcal{M}}
\renewcommand{\O}{\mathcal{O}}
\newcommand{\K}{\mathcal{K}}

\newcommand{\Spec}{\textup{Spec}\,}

\newcommand{\Gr}{\textup{Gr}}
\newcommand{\Fl}{\textup{Fl}}

\newcommand{\semiinf}{\mathfrak{C}^{\frac{\infty}{2}}}

\DeclareMathOperator{\Hom}{Hom}

\DeclareMathOperator{\Ker}{Ker}
\DeclareMathOperator{\Ind}{Ind}

\newtheorem{Theorem}{Theorem}[section]
\newtheorem{Proposition}[Theorem]{Proposition}
\newtheorem{Q-Thm}[Theorem]{Quasi-Theorem}
\newtheorem{Lemma}[Theorem]{Lemma}
\newtheorem{Corollary}[Theorem]{Corollary}
\newtheorem{Conjecture}[Theorem]{Conjecture}

\theoremstyle{definition}
\newtheorem{Definition}{Definition}[section]

\theoremstyle{remark}

\title[Semi-infinite cohomology and Kazhdan-Lusztig equivalence]{Semi-infinite cohomology and Kazhdan-Lusztig equivalence at positive level}
\author{Chia-Cheng Liu}
\address{Department of Mathematics\\
University of Toronto\\
Ontario, Canada}
\email{ccliu@math.toronto.edu}

\begin{document}

\maketitle
\begin{abstract}
	A positive level Kazhdan-Lusztig functor is defined using Arkhipov-Gaitsgory duality for affine Lie algebras. The functor sends objects in the DG category of $G(\O)$-equivariant positive level affine Lie algebra modules to objects in the DG category of modules over Lusztig's quantum group at a root of unity. We prove that the semi-infinite cohomology functor for positive level modules factors through the Kazhdan-Lusztig functor at positive level and the quantum group cohomology functor with respect to the positive part of Lusztig's quantum group.
\end{abstract}

\section{Introduction}

\subsection{Basic setup} \label{basic_setup}
Let $G$ be a complex reductive algebraic group, $B$ (resp. $B^-$) the Borel (resp. opposite Borel) subgroup, $N$ (resp. $N^-$) the unipotent radical in $B$ (resp. $B^-$), and $T = B \cap B^-$ the maximal torus. Let $\g, \b, \n, \b^-, \n^-, \t$ be the corresponding Lie algebras. Let $W$ be the Weyl group of $G$, and $h^{\vee}$ the dual Coxeter number of $G$. Denote by $\check{G}$ the Langlands dual of $G$.

Throughout this paper there is an important parameter $\kappa \in \C^{\times}$, called the \textit{level}.
The level $\kappa$ is called \textit{negative} if $\kappa + h^{\vee} \notin \Q^{\geqslant 0}$,  \textit{positive} if $\kappa + h^{\vee} \notin \Q^{\leqslant 0}$, and \textit{critical} when $\kappa = \kappa_{\textup{crit}}:= -h^\vee$. Given a positive level $\kappa$, the reflected level $\kappa' := -\kappa -2h^{\vee}$ is negative. Clearly, if $\kappa$ is both positive and negative, then $\kappa$ is irrational.

All objects of algebro-geometric nature will be over the base field $\C$. Unless specified otherwise, by category we mean a \textit{DG category}, i.e. an accessible stable $\infty$-category enriched over $\textup{Vect}$, the unit DG category of complexes of $\C$-vector spaces. We denote the heart of the $t$-structure in a DG category $\mathsf{C}$ by $\mathsf{C}^\heartsuit$.

Denote by $\O$ the ring of functions on the formal disk at a point, and by $\K$ the ring of functions on the punctured formal disk. The $\C$-points of $\K$ after taking a coordinate $t$ is the Laurent series ring $\C((t))$, and that of $\O$ becomes its ring of integers $\C[[t]]$. For a $\C$-vector space $V$ we write $V(\K)$ for $V((t)) \equiv V \otimes \C((t))$. For a scheme $Y$ we write $Y(\K)$ as the formal loop space of $Y$, whose $\C$-points are $\textup{Maps}(\Spec \C((t)), Y)$. Similarly we define $V(\O)$ and $Y(\O)$.

Consider the loop group $G(\K)$ and its subgroup $G(\O)$. Define the evaluation map $\textup{ev}: G(\O) \to G$ by $t \mapsto 0$. The Iwahori subgroup of $G(\O)$ is defined as $I :=\textup{ev}^{-1}(B)$. We also define $I^0 := \textup{ev}^{-1}(N)$.

\subsection{The main result}
We study the semi-infinite cohomology of modules over affine Lie algebras. Our main result is a formula (Theorem \ref{*thm}) which relates the semi-infinite cohomology to the quantum group cohomology, as the modules over affine Lie algebras are linked to modules over quantum groups via the Kazhdan-Lusztig tensor equivalence.

The significance of this formula is twofold:
\begin{itemize}
	\item[1.] Integrated with the factorization structures appearing naturally in these objects, the formula paves the way for an alternative proof of the Kazhdan-Lusztig tensor equivalence, which is widely considered overly technical. Moreover, this new approach is valid for any non-critical level $\kappa$, whereas the original Kazhdan-Lusztig equivalence was only developed for negative levels. Hence our result indicates how to generalize the Kazhdan-Lusztig equivalence to arbitrary (non-critical) level conjecturally. We will further elaborate on factorization Kazhdan-Lusztig equivalences later in the Introduction (Section \ref{sec:intro_fact}).
	
	\item[2.] The formula fits in the recent progress on the quantum local geometric Langlands correspondence. The approach, proposed by D. Gaitsgory and J. Lurie, to the correspondence is to relate the Kac-Moody brane and the Whittaker brane by passing both to the quantum group world. The correspondence is ``quantum" in the sense that we should have a family of correspondences parametrized by a non-critical level $\kappa$, with the critical level being the degenerate case. Our formula provides the bridge from the Kac-Moody brane to the quantum group world for any (positive) level $\kappa$. We refer the readers to the talk notes of the Winter School on Local Geometric Langlands Theory \cite{Program} for more information. 
\end{itemize}

\subsection{Semi-infinite cohomology}
The semi-infinite cohomology was first introduced by B. Feigin in \cite{feigin1984}, as the mathematical counterpart of the BRST quantization in theoretical physics. While the notion of semi-infinite cohomology was later generalized to broader settings \cite{Voronov1993, Arkhipov1996, Arkhipov1997, Bezru2010}, our study stays within the original framework, namely, that for the affine Lie algebras.

Purely in terms of algebra, the setting includes a vector space $M$, a (finite-dimensional) Lie algebra $V$, and a Lie algebra action of $V(\K)$ on $M$. Note that the Lie bracket on $V(\K)$ is given by $$[v_1\otimes f(t), v_2 \otimes g(t)]:= [v_1,v_2]\otimes (f(t) \cdot g(t)).$$
We can roughly describe the semi-infinite cohomology with respect to $V(\K)$ as follows: we first take the Lie algebra cohomology along $V(\O)$, and then take the Lie algebra homology along $V(\K)/V(\O)$. The resulting complex in $\textup{Vect}$ is denoted by $\semiinf(V(\K),M)$.

We will consider a finite-dimensional reductive Lie algebra $\g$ over $\C$, and the semi-infinite cohomology with respect to certain Lie subalgebras of $\g(\K)$. The modules we apply the semi-infinite cohomology to, however, will be modules over the \textit{affine Lie algebra} associated to the loop algebra $\g(\K)$. Recall that an affine Lie algebra is a central extension of $\g(\K)$ by the central part $\C\textbf{1}$, with the extension determined by specifying a complex parameter $\kappa$. We obtain the notion of modules over $\hat{\g}_{\kappa}$, the affine Lie algebra at level $\kappa$, by requiring that $\textbf{1}$ always acts by the number 1. The semi-infinite cohomology with respect to $\n(\K)$ of such modules comes naturally with an action of the Heisenberg algebra $\hat{\t}$, which is a central extension of $\t(\K)$ by $\C\textbf{1}$.

A key feature here is that the semi-infinite cohomology introduces a canonical level shift, called the \textit{Tate shift}. More precisely, the semi-infinite cohomology $\semiinf(\n(\K), M)$ of a $\hat{\g}_{\kappa}$-module $M$ turns out to be a module over $\hat{\t}_{\kappa\textup{+shift}}$, the central extension whose 2-cocycle is determined by $\kappa - \kappa_{\textup{crit}}$. This is the true reason that we regard $\kappa_{\textup{crit}}$ as the point of origin when introducing the terminology of positive and negative level.

The abelian category of smooth representations of the Heisenberg algebra $\hat{\t}$ is semi-simple, whose simple objects are called the Fock modules $\pi_{\lambda}$ parametrized by the weights $\Hom(T, \mathbb{G}_m)$ \cite[Section 9.13]{Kac}. Therefore we often take the multiplicity of $\pi_{\mu}$ in $\semiinf(\n(\K), M)$ for each weight $\mu$, and call the resulting functor the $\mu$-component of the semi-infinite cohomology functor. This will appear on one side of our formula.

\subsection{The formula at irrational level}
At the level of abelian categories, we can model the semi-simple category $\g\textup{-mod}^{\textup{f.d.}}$of finite-dimensional $\g$-modules by the category $\hat{\g}_{\kappa}\textup{-mod}^{G(\O)}$ of $G(\O)$-integrable representations of $\hat{\g}_{\kappa}$ with $\kappa$ an irrational number. Recall that the complete set of finite-dimensional irreducible representations of a finite-dimensional simple Lie algebra $\g$ is given by $\lbrace V_{\lambda}: \lambda \textup{ dominant integral weights}\rbrace$. The category $\hat{\g}_{\kappa}\textup{-mod}^{G(\O)}$ is semi-simple with simple objects given by the \textit{Weyl modules} $\mathbb{V}_\lambda^\kappa$, and the canonical equivalence $\g\textup{-mod}^{\textup{f.d.}} \simeq \hat{\g}_{\kappa}\textup{-mod}^{G(\O)}$ is induced by $V_\lambda \mapsto \mathbb{V}_\lambda^\kappa$.

On the other hand, the quantum group $U_q(\g)$ associated to $\g$ is known to have almost identical representation theory as $\g$ when the quantum parameter is not a root of unity. The formula will connect the aforementioned semi-infinite cohomology functor with the quantum group cohomology complex $\textup{C}^\bullet(U_q(\n),\mathcal{M})$ of $U_q(\g)$-module $\mathcal{M}$.

The $\hat{\g}_{\kappa}$-modules when $\kappa$ is irrational and the $U_q(\g)$-modules when $q$ is not a root of unity are related by the \textit{Kazhdan-Lusztig equivalence} at irrational level \cite{KazLus}. This is an equivalence of abelian categories
\begin{equation*}
\textup{KL}^{irr}_G: \hat{\g}_{\kappa}\textup{-mod}^{G(\O)} \overset{\sim}{\longrightarrow} U_q(\g)\textup{-mod}
\end{equation*}
which sends the simple object $\mathbb{V}_\lambda^\kappa$ to the simple $U_q(\g)$-module $\VV_\lambda$ of highest weight $\lambda$.

The formula at irrational level is the isomorphism of complexes of vector spaces
\begin{equation} \label{intro:irr_formula}
\semiinf(\n(\K), M)^\mu \cong \textup{C}^\bullet(U_q(\n), \textup{KL}^{irr}_G(M))^\mu
\end{equation}
for any weight $\mu$. An algebraic proof of \eqref{intro:irr_formula} is given in Section \ref{sec:comm_diag}.

\subsection{Rational level}
Situations are drastically more complicated at rational levels. First, the abelian category $\hat{\g}_{\kappa}\textup{-mod}^{G(\O)}$ is no longer semi-simple when $\kappa$ is rational. Second, the theory starts to bifurcate into the positive level case and the negative level case, and the Kazhdan-Lusztig equivalence only covers the negative one.

We deal with the negative rational level case first. We have the Kazhdan-Lusztig equivalence at negative level, which relates the negative level $\kappa'$ and the quantum parameter $q$ by
$$q = \textup{exp}(\frac{\pi \sqrt{-1}}{\kappa' - \kappa_{\textup{crit}}}).$$
Clearly $q$ is now a root of unity. To add further complication, there are more than one variants of quantum groups at a root of unity: the Lusztig form $U_q^{\textup{Lus}}$, the Kac-De Concini form $U_q^{\textup{KD}}$, and the small quantum group $\u_q$. They fit into the following sequence
\begin{equation} \label{intro:quantum_seq}
U_q^{\textup{KD}} \twoheadrightarrow \u_q \hookrightarrow U_q^{\textup{Lus}}.
\end{equation}
The Kazhdan-Lusztig equivalence in this case is $\textup{KL}^{\kappa'}_G: \hat{\g}_{\kappa'}\textup{-mod}^{G(\O)} \overset{\sim}{\longrightarrow} U^{\textup{Lus}}_q(\g)\textup{-mod}$, which sends Weyl modules to the so-called \textit{quantum Weyl modules}.
The (conjectural) formula at negative rational level is
\begin{equation} \label{intro:neg_formula}
\semiinf(\n(\K), M)^\mu \cong \textup{C}^\bullet(U^{\textup{KD}}_q(\n), \textup{KL}^{\kappa'}_G(M))^\mu
\end{equation} for $M$ in $\hat{\g}_{\kappa'}\textup{-mod}^{G(\O)}$.

Now we consider the positive rational level case. There is a duality between negative and positive level modules, due to Gaitsgory and Arkhipov \cite{ArkhGai}. We denote the duality functor by
$$\mathbb{D}_{G(\O)}:\hat{\g}_{\kappa'}\textup{-mod}^{G(\O)} \overset{\sim}{\longrightarrow} \hat{\g}_{\kappa}\textup{-mod}^{G(\O)},$$
where $\kappa$ is positive (and so $\kappa'$ is negative). A Kazhdan-Lusztig type functor at positive level can be defined using the duality as
\begin{equation*}
\textup{KL}^{\kappa}_G := \mathbb{D}^q \circ \textup{KL}^{\kappa'}_G \circ \mathbb{D}_{G(\O)}^{-1},
\end{equation*}
where $\mathbb{D}^q: U_q^{\textup{Lus}}(\g)\textup{-mod} \overset{\sim}{\longrightarrow} U_q^{\textup{Lus}}(\g)\textup{-mod}$ is the contragredient duality for modules over quantum groups.

The crucial difference here is that the functor $\textup{KL}^{\kappa}_G$ at positive level only makes sense in the derived world. This is due to the fact that the duality functor $\mathbb{D}_{G(\O)}$ is only defined on derived categories and does not preserve the heart of the $t$-structures. However, the functor $\textup{KL}^{\kappa}_G$ does send Weyl modules to quantum Weyl modules.

The formula at positive rational level (Theorem \ref{*thm}) is
\begin{equation} \label{intro:pos_formula}
\semiinf(\n(\K), M)^\mu \cong \textup{C}^\bullet(U^{\textup{Lus}}_q(\n), \textup{KL}^{\kappa}_G(M))^\mu
\end{equation} for $M$ in $\hat{\g}_{\kappa}\textup{-mod}^{G(\O)}$.
Note that the duality involved in the definition of $\textup{KL}^{\kappa}_G$ has the effect of swapping the quantum groups in the sequence \eqref{intro:quantum_seq}. This is just another incarnation of the fact that the Verdier duality swaps the standard (!-) and costandard (*-) objects, while preserving the intermediate (!*-) objects. Indeed, we can realize the positive and negative level categories geometrically as D-modules on the affine flag variety by the Kashiwara-Tanisaki localization, and the functor $\mathbb{D}_{G(\O)}$ corresponds to the Verdier dual.

The main result of this paper is a proof of the positive level formula \eqref{intro:pos_formula}, whereas the negative level formula \eqref{intro:neg_formula} is the subject of \cite{quantumsemiinf}, and is still a conjecture with partial results obtained. We now explain the idea of proof at positive level, which follows the same pattern as in the work \textit{loc. cit.} by Gaitsgory.

The quantum Frobenius gives rise to a short exact sequence of categories
\begin{equation} \label{intro:q_Frob}
0 \to \textup{Rep}(\check{B}) \to U^{\textup{Lus}}_q(\b)\textup{-mod} \to \u_q(\b)\textup{-mod} \to 0.
\end{equation}
The strategy is to first characterize the cohomology functor on the Kac-Moody side that corresponds to $\textup{C}^\bullet(\u_q(\n),-)^\mu$, and then pass to $\textup{C}^\bullet(U^{\textup{Lus}}_q(\n),-)^\mu$ using the sequence \eqref{intro:q_Frob}. For this purpose we construct the !*-generalized semi-infinite cohomology functor $\semiinf_{!*}(\n(\K),-)$, which is made possible by the recent discovery \cite{semiinf} of a non-standard $t$-structure on the category $\textup{D-mod}(\Gr_G)^{N(\K)}$ of $N(\K)$-equivariant D-modules on the affine Grassmannian $\Gr_G$, and along with the discovery the construction of a semi-infinite intersection cohomology (IC) object $\textup{IC}^{\frac{\infty}{2}}$ in $\textup{D-mod}(\Gr_G)^{N(\K)}$.

We prove (Theorem \ref{!*thm}):
\begin{equation} \label{intro:!*_formula}
\semiinf_{!*}(\n(\K), M)^\mu \cong \textup{C}^\bullet(\u_q(\n), \textup{KL}^{\kappa}_G(M))^\mu.
\end{equation}
Identifying the coweight lattice as a sublattice of the weight lattice (depending on the parameter $\kappa$), we
consider all $\check{\nu}$-components $\semiinf_{!*}(\n(\K), M)^{\check{\nu}}$ at the same time; i.e. we take the direct sum over all coweights $\check{\nu}$. The resulting object acquires a $\check{B}$-action, and its $\check{B}$-invariants is precisely $\semiinf(\n(\K), M)^0$ by the theory of Arkhipov-Bezrukavnikov-Ginzburg \cite{ABG}. On the quantum group side this procedure produces $\textup{C}^\bullet(U^{\textup{Lus}}_q(\n), \textup{KL}^{\kappa}_G(M))^0$ by the sequence \eqref{intro:q_Frob}. Thus we have established
$$\semiinf(\n(\K), M)^0 \cong \textup{C}^\bullet(U^{\textup{Lus}}_q(\n), \textup{KL}^{\kappa}_G(M))^0,$$
and for general $\mu$ the same procedure applies to the identity with a $\mu$-shift.

\subsection{Factorization} \label{sec:intro_fact}
The category $\hat{\g}_{\kappa}\textup{-mod}^{G(\O)}$ has a non-trivial braided monoidal structure constructed by Kazhdan and Lusztig via the Knizhnik-Zamolodchikov equations \cite{KazLus}. The most remarkable part of the Kazhdan-Lusztig equivalence is that it is an equivalence respecting the braided monoidal structures, where the braided monoidal structure on the quantum group side is given by the R-matrix. 

As early as in the works of Felder-Wieczerkowski \cite{felder}, Schechtman \cite{Schecht} and Schechtman-Varchenko \cite{SV1, SV2}, mathematicians realized that the R-matrix of a quantum group is related to topological factorizable objects on certain inductive limit of configuration spaces. The works in this direction culminated in \cite{BFS} where Bezrukavnikov-Finkelberg-Schechtman established a topological realization of the category of modules over the small quantum group in terms of factorizable sheaves.

On the other hand, Khoroshkin-Schechtman \cite{KS1, KS2} constructed the algebro-geometric factorizable objects which they call the factorizable D-modules. Their construction gives an algebro-geometric realization of the category $\hat{\g}_{\kappa}\textup{-mod}^{G(\O)}$ for $\kappa$ irrational (more precisely, Drinfeld's tensor category of $\g$-modules), and via the Riemann-Hilbert correspondence it corresponds to the BFS factorizable sheaves.

Therefore, after the respective realization as factorizable objects, the Kazhdan-Lusztig equivalence at irrational level is deduced from the Riemann-Hilbert correspondence, which clearly preserves the factorization structures retaining the braided monoidal structures in the original categories.

The general philosophy \cite[Section 1.8 and 1.9]{Ras} is that there should be a correspondence between \textit{factorization categories} and braided monoidal categories. It is hence expected that a factorization form of the Kazhdan-Lusztig equivalence (for not just the irrational levels) exists.

We digress temporarily to discuss the notion of strong group actions on categories. We say a category $\mathsf{C}$ is acted on strongly by a group $H$ if $\mathsf{C}$ is a module category of $\textup{D-mod}(H)$. We can twist the category $\textup{D-mod}(H)$ by a multiplicative $\mathbb{G}_m$-gerbe on $H$, which is equivalent to the data of a central extension $\hat{\mathfrak{h}}$ of the Lie algebra $\mathfrak{h}=\textup{Lie}(H)$, with a lift of the adjoint action of $H$ on $\mathfrak{h}$ to $\hat{\mathfrak{h}}$, c.f. \cite{NickGai}.

In the case of the loop group $G(\K)$ of a reductive group $G$, corresponding to the affine Kac-Moody extension $\hat{\g}_\kappa$ we have the $\kappa$-twisted category $\textup{D-mod}_\kappa(G(\K))$. A category is acted on strongly by $G(\K)$ at level $\kappa$ if it is a module category of $\textup{D-mod}_\kappa(G(\K))$.

From the algebro-geometric perspective, factorization structures arise naturally from (strong) actions of a loop group \cite{GaiFactCat}. The category $\hat{\g}_{\kappa}\textup{-mod}^{G(\O)}$ is acted on strongly by $G(\K)$ at level $\kappa$, which essentially comes from the action of the loop algebra $\g(\K)$ on $\hat{\g}$. Then we obtain the factorization category $(\hat{\g}_{\kappa}\textup{-mod}^{G(\O)})_{\textup{Ran}(X)}$. 

The main difficulty to achieve a factorization Kazhdan-Lusztig equivalence lies in the quantum group side. Since the braided monoidal structure for quantum group modules is of topological nature, the most natural factorization structure in this case should be of topological flavor. The theory of \textit{topological factorization categories} was developed by J. Lurie in terms of algebras over the little disks operad in the $(\infty,2)$-category of DG categories \cite{Lurie2}. Tautologically, the topological factorization category associated to the braided monoidal category $\textup{Rep}_q(T)$ of representations of the quantum torus is $\textup{Shv}_q(\Gr_{\check{T},\textup{Ran}(X)})$, the constructible sheaves on the Beilinson-Drinfeld Grassmannian of $\check{T}$, twisted by a factorizable gerbe specified by $q$. Now the question is, how to explicitly build the topological factorization category associated to $U^{\textup{Lus}}_q(\g)\textup{-mod}$?

Owing to Lurie's theory, the answer is positive, if we replace the quantum group $U^{\textup{Lus}}_q(\g)$ by the small quantum group $\u_q(\g)$, or by a mixed quantum group $U^{+\textup{Lus},-\textup{KD}}_q(\g)$ which has positive part in Lusztig form and negative part in Kac De Concini form. Nevertheless, we are still unable to construct explicitly the topological factorization category $\textup{Fact}(U^{\textup{Lus}}_q(\g)\textup{-mod})$ associated to $U^{\textup{Lus}}_q(\g)\textup{-mod}$.

We can, however, modify the factorization category associated to $U^{+\textup{Lus},-\textup{KD}}_q(\g)\textup{-mod}$ in algebraic terms to get a factorization category $\textup{Fact}(U^{\frac{1}{2}}_q(\g)\textup{-mod})$ that contains $\textup{Fact}(U^{\textup{Lus}}_q(\g)\textup{-mod})$ as a full subcategory. Even better, under the Riemann-Hilbert correspondence, the category $\textup{RH}(\textup{Fact}(U^{\frac{1}{2}}_q(\g)\textup{-mod}))$ is the natural recipient of a certain factorizable functor, called the \textit{Jacquet functor}, from $(\hat{\g}_{\kappa}\textup{-mod}^{G(\O)})_{\textup{Ran}(X)}$.

The factorization Kazhdan-Lusztig equivalence can now be formulated as
\begin{Conjecture} \label{intro:conj}
	The Jacquet functor is fully faithful, with its essential image identified with $$\textup{Fact}(U^{\textup{Lus}}_q(\g)\textup{-mod})$$ under the Riemann-Hilbert correspondence.
\end{Conjecture}

The upshot is that, in the positive level case, the Jacquet functor is precisely the semi-infinite cohomology functor. Our formula \eqref{intro:pos_formula} therefore plays an instrumental role in tackling Conjecture \ref{intro:conj}. Moreover, as the definition of the semi-infinite cohomology is purely algebraic, one sees in this characterization that the transcendental nature of the Kazhdan-Lusztig equivalence exactly comes from that of the Riemann-Hilbert correspondence.

The categories $\u_q(\g)\textup{-mod}$, $U^{+\textup{Lus},-\textup{KD}}_q(\g)\textup{-mod}$ and $U^{\frac{1}{2}}_q(\g)\textup{-mod}$ will be defined concretely in Section \ref{sec:quantum_groups} for completeness.

We remark that, at negative level, the definition of the Jacquet functor involves the !-generalized semi-infinite cohomology functor, which is introduced and studied in \cite{quantumsemiinf}.

\subsection{Structure}
The paper is organized as follows.

In Section \ref{chapter:alg_constr} and Section \ref{chapter:geo_constr} we recall standard constructions and results from Lie theory and geometric representation theory. In particular, in Section \ref{sec:affine_lie_algebra} we define the duality functor between negative level and positive level modules, and calculate the image of affine Verma modules and Weyl modules. In Section \ref{sec:KL_positive} we define the positive level Kazhdan-Lusztig functor by means of the duality functor.

Section \ref{chapter:wakimoto} introduces the Wakimoto modules. In particular, we prove Theorem \ref{th:WakiVer}, which identifies the type $w_0$ Wakimoto module with the dual affine Verma module of the same highest weight under certain conditions on the highest weight and the level. In Section \ref{sec:rel_to_semiinf} we present two formulas which compute semi-infinite cohomology using Wakimoto modules.

Section \ref{chapter:semiinf_pos} is the main thrust of the paper.
In Section \ref{sec:gen_semiinf_functor} we introduce the generalized semi-infinite cohomology functor at positive level, and define the semi-infinite IC object $\textup{IC}^{\frac{\infty}{2},-}$ used in defining the !*-generalized functor. 
Section \ref{sec:formulas} states our main results, the formula for the !*-functor (Theorem \ref{!*thm}) and the formula for the original semi-infinite functor (Theorem \ref{*thm}).

Section \ref{chapter:irr_lvl} gives an algebraic proof of the main formula when the level is assumed irrational.

\subsection{Acknowledgements}
I would like to thank my advisors Joel Kamnitzer and Alexander Braverman, for countless stimulating discussions during which they shared with me their broad mathematical knowledge and insights.

I am hugely indebted to Dennis Gaitsgory, who formulated and introduced to me the problem which this paper addresses, and patiently taught me all the key concepts required to carry out this work.

I also benefited from conversations with Dylan Buston, Justin Campbell, Dinakar Muthiah, Sam Raskin, Nick Rozenblyum, Alex Weekes, Philsang Yoo, Yifei Zhao, and Xinwen Zhu.

\section{Preparation: algebraic constructions} \label{chapter:alg_constr}
\subsection{Root datum}
Recall notations from Section \ref{basic_setup} for algebraic groups and their Lie algebras. 
Let $\Lambda$ (resp. $\check{\Lambda}$) be the weight (resp. coweight) lattice of $G$. Then by definition $\check{\Lambda}$ (resp. $\Lambda$) is the weight (resp. coweight) lattice of $\check{G}$. Write $\Lambda^+$ (resp. $\check{\Lambda}^+$) for the set of dominant weights (resp. coweights).
Let $R$, $R^+$, and $\Pi$ denote the set of roots, positive roots, and simple roots of $G$, respectively. Let $\rho$ denote the half sum of all positive roots.

We have the standard invariant bilinear form on $\g$
$$(\cdot,\cdot)_{\textup{st}}: \g \otimes \g \to \C,$$ which restricts to a form on $\t$ and thus on the lattice $\check{\Lambda} \subset \t$.
We also have the natural pairing $\langle \cdot, \cdot\rangle$ between $\t^*$ and $\t$, which restricts to
$$\langle \cdot, \cdot\rangle : \Lambda \otimes \check{\Lambda} \to \Z$$ on the lattices. For each simple root $\alpha_i$ and coroot $\check{\alpha}_i$, let $d_i \in \lbrace 1,2,3 \rbrace$ be the integer such that $(\check{\alpha}_i, \check{\mu})_{\textup{st}} = d_i\langle \alpha_i, \check{\mu}\rangle$. Then there is an induced form on $\Lambda$, also denoted by $(\cdot,\cdot)_{\textup{st}}$ when no confusion can arise, characterized by the relations $(\mu, \alpha_i)_{\textup{st}} = d_i^{-1}\langle\mu , \check{\alpha}_i\rangle$ for all $i$.

For a number $\kappa \in \C^{\times}$, we set $$(\cdot,\cdot)_{\kappa}:= \kappa (\cdot,\cdot)_{\textup{st}}: \g \otimes \g \to \C.$$
We then have the corresponding form $(\cdot,\cdot)_{\kappa}$ on $\Lambda$ which satisfies
$$\frac{(\check{\alpha}_i,\check{\alpha}_j)_{\kappa}}{(\check{\alpha}_i,\check{\alpha}_i)_{\kappa}} \cdot d_i = d_j \cdot \frac{(\alpha_i,\alpha_j)_{\kappa}}{(\alpha_i,\alpha_i)_{\kappa}}.$$

Define the isomorphism $\phi_{\kappa}: \t \to \t^*$ by the relation $(\lambda, \phi_{\kappa}(\check{\mu}))_{\kappa - \kappa_{\textup{crit}}} = \langle \lambda, \check{\mu}\rangle$. 
We will abuse the notation by writing $\check{\mu}$ in place of $\phi_{\kappa}(\check{\mu})$ when both weights and coweights are present in an expression. For example, we will write $\lambda + \check{\mu}$ instead of $\lambda+\phi_{\kappa}(\check{\mu})$.
When $\kappa$ is a rational number such that $\phi_{\kappa}: 
\check{\Lambda} \to \Lambda$, let $\phi_T: T \to \check{T}$ be the map induced from $\phi_{\kappa}$.

\subsection{Affine Lie algebras} \label{sec:affine_lie_algebra}
The affine Lie algebra $\hat{\g}_\kappa$ at level $\kappa$ is a central extension of $\g (\K)$ by the 1-dimensional trivial module $\C \textbf{1}$, with Lie bracket defined by the 2-cocycle $$f, g \mapsto \textrm{Res}_{t=0} (f\,,\frac{dg}{dt} )_{\kappa}, \quad f, g \in \g(\K).$$
We have the Cartan decomposition of the affine Lie algebra $\hat{\g}_\kappa$:
$$\hat{\g}_\kappa = \n (\K) \oplus \hat{\t}_{\kappa} \oplus \n^{-}(\K),$$
where the subalgebra $\hat{\t}_{\kappa}:= \t(\K) \oplus \C \textbf{1}$ is the Heisenberg algebra at level $\kappa$.

We define $\hat{\g}_{\kappa}\textup{-mod}^
\heartsuit$ as the abelian category whose objects are $\hat{\g}_{\kappa}$-modules $M$ where
(1) the central element $\textbf{1}$ acts as the identity, and
(2) each vector $m \in M$ is annhilated by $\g(t^n\C[[t]])$ for some $n \geq 0$. The morphisms are ordinary $\hat{\g}_{\kappa}$-equivaraint maps. The corresponding DG category is denoted by $\hat{\g}_{\kappa}\textup{-mod}$.

The full subcategory $(\hat{\g}_{\kappa}\textup{-mod}^{G(\O)})^
\heartsuit \subset \hat{\g}_{\kappa}\textup{-mod}^
\heartsuit$ consists of those modules whose $\g(\O)$-action comes from a $G(\O)$-action. As explained in \cite[Section 1.2]{GaiKM}, the $G(\O)$-action integrating a given $\g(\O)$-action is unique at the abelian level, but not so at the derived level since higher cohomologies of $G(\O)$ are non-trivial. Consequently, the DG category $\hat{\g}_{\kappa}\textup{-mod}^{G(\O)}$ of $G(\O)$-equivariant (equivalently, $G(\O)$-integrable) $\hat{\g}_{\kappa}$-modules is no longer a full subcategory of the DG category $\hat{\g}_{\kappa}\textup{-mod}$. Nevertheless, one can construct the DG category $\hat{\g}_{\kappa}\textup{-mod}^{G(\O)}$ by ``bootstrapping" from the abelian category $(\hat{\g}_{\kappa}\textup{-mod}^{G(\O)})^
\heartsuit$. The details of this construction appear in Section 2 of \textit{loc. cit.} A different approach to construct $\hat{\g}_{\kappa}\textup{-mod}^{G(\O)}$, using derived algebraic geometry and higher category theory, is given in Section 4 of \textit{loc. cit.}, which might be of interest to the reader.

For a given finite-dimensional representation $V$ of $\g$, we extend it to a module over $\g (\O) \oplus \C \textbf{1}$ by setting the action of $t$ as 0 and the action of the central element $\textbf{1}$ as 1. Then we perform induction to $\hat{\g}_{\kappa}$: $$V^{\kappa} := \Ind^{\hat{\g}_{\kappa}}_{\g (\O) \oplus \C \textbf{1}}V.$$
The resulting $\hat{\g}_{\kappa}$-module $V^{\kappa}$ is obviously $G(\O)$-integrable.

Let $V_{\lambda}$ be the finite-dimensional irreducible representation of $\g$ with dominant integral highest weight $\lambda$. The corresponding $\hat{\g}_\kappa$-module $\mathbb{V}^{\kappa}_{\lambda}:= (V_{\lambda})^{\kappa}$ is called the Weyl module of highest weight $\lambda$. The category $\hat{\g}_{\kappa}\textup{-mod}^{G(\O)}$ is compactly generated by the subcategory of Weyl modules.

Recall the Verma module $M_{\lambda} := \Ind^{\g}_{\b}\C_{\lambda}$ of highest weight $\lambda$ over $\g$. Then we define the affine Verma module as $\mathbb{M}^{\kappa}_{\lambda}:=(M_{\lambda})^{\kappa}$. Equivalently, $\mathbb{M}^{\kappa}_{\lambda} = \Ind^{\hat{\g}_{\kappa}}_{\textup{Lie}(I)}\C_{\lambda}$, where we regard $\C_{\lambda}$ as an $\textup{Lie}(I)$-module via the evalution map. By definition affine Weyl modules are objects in (the heart of) the DG category $\hat{\g}_{\kappa}\textup{-mod}^I$ of $I$-equivariant $\hat{\g}_{\kappa}$-modules.

Let $\C_{\mu}$ be the 1-dimensional $\t$-module of weight $\mu$. Then similar to the definition of $V^{\kappa}$ above we define the Heisenberg module
$$\pi^{\kappa}_{\mu} := \Ind^{\hat{\t}_{\kappa}}_{\t (\O) \oplus \C \textbf{1}} \C_{\mu},
$$
called the Fock module of highest weight $\mu$.

We recall from \cite{ArkhGai} a perfect pairing between negative level and positive level $\hat{\g}$-modules:
$$\langle -,- \rangle : \hat{\g}_{\kappa'} \textup{-mod} \times \hat{\g}_{\kappa} \textup{-mod} \to \textup{Vect}$$
given by $\langle N,M \rangle = \semiinf(\g(\K), N \otimes_{\C} M)$. Here the semi-infinite complex is taken with respect to the lattice $\g(\O) \subset \g(\K)$.
Suppose that a group $K$ acts on $\hat{\g} \textup{-mod}$. Then \begin{equation} \label{eq:equiv_pairing}
\langle -,- \rangle_K: \hat{\g}_{\kappa'}\textup{-mod}^K \times \hat{\g}_{\kappa}\textup{-mod}^K \overset{\langle -,- \rangle}{\longrightarrow} \textup{Vect}^K \overset{H_K^{\bullet}}{\longrightarrow} \textup{Vect}
\end{equation} 
defines a pairing between negative and positive level $K$-equivariant categories.
Let $\mathbb{D}_K: \hat{\g}_{\kappa'} \textup{-mod}^K \to \hat{\g}_{\kappa} \textup{-mod}^K$ be the (contravariant) duality functor which satisfies
\begin{equation} \label{pairing}
\langle A,B \rangle_K = \Hom_{\hat{\g}_{\kappa}} (\mathbb{D}_K A, B).
\end{equation}

Let $H$ be a subgroup of $K$. The pairing for equivariant categories naturally commutes with the forgetful functor $\textrm{obliv}: \hat{\g} \textrm{-mod}^K \to \hat{\g} \textrm{-mod}^H$; i.e. the following diagram commutes:
\begin{equation*}
\xymatrix @C=6pc{
	\hat{\g}_{\kappa'}\textup{-mod}^K \ar[d]_{\textup{obliv}} \ar[r]^{\mathbb{D}_K} & \hat{\g}_{\kappa}\textup{-mod}^K \ar[d]^{\textup{obliv}} \\
	\hat{\g}_{\kappa'}\textup{-mod}^H \ar[r]^{\mathbb{D}_H} & \hat{\g}_{\kappa}\textup{-mod}^H
}
\end{equation*}

Recall that the *-averaging functor $\textrm{Av}_*^{K/H}$ (resp. !-averaging functor $\textrm{Av}_!^{K/H}$) is defined as the right (resp. left) adjoint to the forgetful functor $\textup{obliv}: \hat{\g} \textup{-mod}^K \to \hat{\g} \textup{-mod}^H$. Then it follows from the definitions that \begin{equation} \label{duality_average}
\mathbb{D}_K \circ \textrm{Av}_*^{K/H} \simeq \textrm{Av}_!^{K/H} \circ \mathbb{D}_H.
\end{equation}

We end this section by two simple but important calculations of the duality functor.

\begin{Lemma} \label{D_Verma}
	For any weight $\mu$, $\mathbb{D}_I(\mathbb{M}^{\kappa'}_{\mu}) \cong \mathbb{M}^{\kappa}_{-\mu + 2\rho}[\dim(G/B)]$.
\end{Lemma}
\begin{proof}[Proof 1]
	From \cite[Section 2.2.8]{ArkhGai}, we have $$\mathbb{D}_I(\Ind^{\hat{\g}_{\kappa'}}_{I}\C_{\nu}) \cong \Ind^{\hat{\g}_{\kappa}}_{I}\C_{-\nu} \otimes \textup{det. rel.}(\textup{Lie}(I), \g(\O)).$$ Note that $$\textup{det. rel.}(\textup{Lie}(I), \g(\O)) \cong \textup{det}(\g/\b)^*$$ is the graded 1-dimensional $B$-module of weight $2\rho$ concentrated at degree $-\dim(G/B)$.
\end{proof}
\begin{proof}[Proof 2]
	For $L \in \hat{\g}_{\kappa}\textup{-mod}^I$, we have 
	\begin{equation*}
	\begin{split}
	\langle \mathbb{M}^{\kappa'}_{\mu}, L\rangle_I \cong H^{\bullet}_I(\semiinf(\g(\K), (\Ind^{\hat{\g}_{\kappa'}}_I \C_{\mu}) \otimes L))
	\cong \Hom_{I\textup{-mod}}(\C_{-\mu}\otimes \textup{det}(\g(\O)/\textup{Lie}(I))^*, L)\\
	\cong \Hom_{I\textup{-mod}}(\C_{-\mu+2\rho}[\dim(G/B)], L)
	\cong \Hom_{\hat{\g}_{\kappa}\textup{-mod}^I}(\mathbb{M}^{\kappa}_{-\mu+2\rho}[\dim(G/B)], L).
	\end{split}
	\end{equation*}
	Then the assertion follows from \eqref{pairing}.	
\end{proof}
As this lemma reveals, the duality functor $\mathbb{D}$ (and its equivariant versions) does not preserve the usual $t$-structures. However, it does preserve the compact objects:

\begin{Lemma} \label{D_Weyl}
	For $\lambda \in \Lambda^+$, $\mathbb{D}_{G(\O)}(\mathbb{V}^{\kappa'}_{\lambda}) \cong \mathbb{V}^{\kappa}_{-w_0(\lambda)}$.
\end{Lemma}
\begin{proof}
	Similar to the proofs of Lemma \ref{D_Verma}. Note that the lowest weight of $V_{\lambda}$ is $w_0(\lambda)$.
\end{proof}

\subsection{Quantum groups}\label{sec:quantum_groups}
Let $\ell$ be a sufficiently large positive integer divisible by all $d_i$'s, and put $\ell_i := \ell /d_i$.
Recall the following variants of quantum groups associated to $\g$ and a primitive $\ell$-th root of unity $q$ (c.f. \cite{ArkhGai2} and \cite{Chari}):
\begin{itemize}
	\item The Drinfeld-Jimbo quantum group $\mathbf{U}_v$, generated by Chavelley generators $E_i$, $F_i$ and $K_t$ for $i=1, ..., \textup{rank}\g$ and $t \in T$ over $\C(v)$, the rational functions in $v$, subject to a list of relations.
	\item The (Lusztig's) big quantum group $U^{\textup{Lus}}_q(\g) \equiv U_q(\g)$: Take $R := \C[v, v^-1]_{(v-q)} \subset \C(v)$. The algebra $U_q(\g)$ is the specialization to $v=q$ of the $R$-subalgebra of $\mathbf{U}_v$ generated by $E_i$, $F_i$, $K_t$, and divided powers $E^{(\ell_i)}_i$, $F^{(\ell_i)}_i$.
	\item The Kac-De Concini quantum group $U_q^{\textup{KD}}(\g)$: the specialization to $v=q$ of the $R$-subalgebra of $\mathbf{U}_v$ generated by $E_i$, $F_i$, $K_t$, and $\dfrac{K_i - K^{-1}_i}{v - v^{-1}}$.
	\item The small quantum group $\u_q(\g)$, defined as the $\C$-subalgebra of $U_q(\g)$ generated by (i) $K_iE_i$, $F_i$, and $K_t$ for $t\in \Ker(\phi_T)$ when $\ell$ is even, or (ii) $E_i$, $F_i$, and $K_i$ when $\ell$ is odd.
\end{itemize}
Let $A$ be one of the above quantum groups. $A\textup{-mod}$ denotes the category of finite-dimensional modules over the Hopf algebra $A$.

We define the quantum Frobenius functor $$\textup{Fr}_q: \textup{Rep}(\check{G}) \to U^{\textup{Lus}}_q(\g)\textup{-mod}$$ as follows: given $V \in \textup{Rep}(\check{G})$, let $K_t$ act on $V$ according to $\phi_T: T \to \check{T}$, and let $E^{(\ell_i)}_i$, $F^{(\ell_i)}_i$ act as Chevalley generators $e_i$, $f_i$ of $U(\check{\g})$, whereas $E_i$ and $F_i$ act by 0.
Then the small quantum group $\u_q(\g)$ is the Hopf subalgebra of $U^{\textup{Lus}}_q(\g)$ universal with respect to the property that $\u_q(\g)$ acts trivially on $\textup{Fr}_q(V)$ for $V \in \textup{Rep}(\check{G})$. Namely, we have the following exact sequence of categories
\begin{equation}
0 \to \textup{Rep}(\check{G}) \to U^{\textup{Lus}}_q(\g)\textup{-mod} \to \u_q(\g)\textup{-mod} \to 0.
\end{equation}

We can also define quantum groups for $\b$, $\n$, $\t$ ...etc. In particular, we will consider the positive part of the various versions of quantum group defined above. In the remainder of this section we give an alternative construction of these quantum groups, where the quantum parameter $q$ and our level $\kappa$ are related in a more transparent way. Our main reference for this construction is \cite{Riche}.

Recall the root lattice $\Z \Pi$ of $G$. We start with a bilinear form $b: \Z \Pi \times \Z \Pi \to \C^{\times}$. Consider the braided monoidal category $\textup{Rep}_q(T_{\textup{ad}})$ of representations of the quantum adjoint torus: the objects are $\Z\Pi$-graded vector spaces, bifunctor the usual tensor product $\otimes$, and braiding operator induced by
\begin{equation} \label{braiding_rep_qT}
x \otimes y \mapsto b(\lambda, \mu)\, y \otimes x
\end{equation}
for $x \in \C_{\lambda}$ and $y \in \C_{\mu}$.
Consider the object $\oplus_{i \in \Pi} \C E_i$ in $\textup{Rep}_q(T_{\textup{ad}})$, where each $E_i$ is in degree $\alpha_i \in \Pi$. Let
$$U_q^{\textup{free}}(\n):= \textup{the graded free associative algebra generated by } \oplus_{i \in \Pi} \C E_i \textup{ in }\textup{Rep}_q(T_{\textup{ad}}).$$
Similarly, we define $U_q^{\textup{free}}(\n^-)$ generated by $\oplus_{i \in \Pi} \C F_i$ with each $F_i$ in degree $-\alpha_i$. Then set
$$U_q^{\textup{cofree}}(\n):= \underset{\lambda \in \Z_{\geq 0}\Pi}{\oplus}((U_q^{\textup{free}}(\n^-))_{-\lambda})^*.$$
We have the canonical bialgebra structure on $U_q^{\textup{free}}(\n)$, which induces the canonical bialgebra structure on $U_q^{\textup{cofree}}(\n)$. We have a canonical bialgebra map $\varphi: U_q^{\textup{free}}(\n) \to U_q^{\textup{cofree}}(\n)$ which sends $E_i$ to $\delta_{F_i}$.

For each $i \in \Pi$, choose $v_i := b(\alpha_i, \alpha_i)^{1/2}$. For arbitrary $i,j \in \Pi$, assume that $b(\alpha_i, \alpha_j) = b(\alpha_j, \alpha_i) = v_i^{\langle \alpha_i, \check{\alpha}_j \rangle}$. For integers $n, m$, define the quantum binomial coefficient
$${\begin{bmatrix}
	n \\
	m
	\end{bmatrix}}_i = \frac{[n]^!_i}{[m]^!_i[n-m]^!_i}
$$ where $[n]_i = \frac{v_i^n - v_i^{-n}}{v_i - v_i^{-1}}$ and $[n]^!_i = \Pi_{s=1}^{n} [s]_i$. The \textit{quantum Serre relation} corresponding to $i,j \in \Pi$ is the element
\begin{equation*}
\underset{p+p'= 1-\langle \alpha_i,\check{\alpha}_j \rangle}{\sum} (-1)^{p'} {\begin{bmatrix}
	1-\langle \alpha_i,\check{\alpha}_j \rangle \\
	p
	\end{bmatrix}}_i E_i^p E_j E_i^{p'}.
\end{equation*}

We define $U_q^{\textup{KD}}(\n)$ as the quotient of $U_q^{\textup{free}}(\n)$ by the quantum Serre relations for all $i,j \in \Pi$. It is verified in \cite{Riche} that the quantum Serre relations are sent to zero under $\varphi$. Hence $\varphi$ induces a map $\tilde{\varphi}: U_q^{\textup{KD}}(\n) \to U_q^{\textup{cofree}}(\n)$. Define $U_q^{\textup{Lus}}(\n) \subset U_q^{\textup{cofree}}(\n)$ as the sub-bialgebra linearly dual to $U_q^{\textup{KD}}(\n^-)$. It is shown in \textit{loc. cit.} that $\tilde{\varphi}$ factors through $U_q^{\textup{Lus}}(\n)$. Finally, the small quantum group $\u_q(\n)$ is defined as the image $\tilde{\varphi}(U_q^{\textup{KD}}(\n))$, regarded as a sub-bialgebra of $U_q^{\textup{Lus}}(\n)$. Therefore we have the following diagram
\begin{equation*}
\xymatrix @C=4pc{ U_q^{\textup{free}}(\n) \ar@{->>}[r] \ar[dd]_{\varphi} & U_q^{\textup{KD}}(\n) \ar[ddl]^{\tilde{\varphi}} \ar[dd]^{\tilde{\varphi}} \ar@{->>}[dr]\\
	& & \u_q(\n) \ar@{_{(}->}[dl]\\
	U_q^{\textup{cofree}}(\n) & U_q^{\textup{Lus}}(\n) \ar@{_{(}->}[l]	
}
\end{equation*}

Various versions of representation category of the full quantum group can now be obtained by applying the notion of relative Drinfeld center to the category of modules over $U_q^{\textup{Lus}}(\n)$. Suppose that $b$ is extended to a bilinear form on $\Lambda$, and recall $\textup{Rep}_q(T)$ the braided monoidal category of $\Lambda$-graded vector spaces, with the braiding operator given by the same formula as \eqref{braiding_rep_qT}. Let $U_q^{\textup{Lus}}(\n)$-mod be the category consisting of objects in $\textup{Rep}_q(T)$ together with action by $U_q^{\textup{Lus}}(\n)$ compatible with the grading.

Consider the category $U_q^{\textup{Lus}}(\b)$-mod whose objects are unions of finite-dimensional $U_q^{\textup{Lus}}(\n)$-submodules. We let $U_q^{\textup{Lus}}(\n)$ act on the usual tensor product of modules by
$$x \cdot (v \otimes v') := \sum b(\deg(v),\deg(x_{(2)})) \, (x_{(1)} \cdot v)\otimes(x_{(2)}\cdot v').
$$
Then both $U_q^{\textup{Lus}}(\n)$-mod and $U_q^{\textup{Lus}}(\b)$-mod are now braided monoidal categories.

We first define the category $U_q^{\textup{+Lus,--KD}}(\g)$-mod of representations over a ``mixed" quantum group, whose positive part is of Lusztig form and negative part is of Kac-De Concini form. Concretely,
the category $U_q^{\textup{+Lus,--KD}}(\g)$-mod is the relative Drinfeld center $\textup{Dr}_{\textup{Rep}_q(T)}(U_q^{\textup{Lus}}(\b)\textup{-mod})$, whose objects are pairs $(M,c)$, where $M$ is in $U_q^{\textup{Lus}}(\b)\textup{-mod}$ and 
$$c_{M,X}: M\otimes X \tilde{\longrightarrow} X \otimes M
$$ are isomorphisms functorial in $X$, such that $c_{M,X\otimes Y} = (\textup{Id}_X \otimes c_{M,Y}) \cdot (c_{M,X} \otimes \textup{Id}_Y)$ and $c_{M,\C_{\nu}}(m \otimes x) = b(\deg(m),\nu) \cdot x \otimes m$ for all $\nu \in \Z\Pi$.

The collection $c$ of functorial isomorphisms of an object $(M,c)$ in $U_q^{\textup{+Lus,--KD}}(\g)$-mod gives rise to a right coaction $\iota_M: M \to M \otimes U_q^{\textup{Lus}}(\n)$ of $U_q^{\textup{Lus}}(\n)$ on $M$, defined by
$$\iota_M (m) := (c_{M, U_q^{\textup{Lus}}(\n)})^{-1}(1\otimes m). 
$$ This induces a left $(U_q^{\textup{Lus}}(\n))^*$-action and hence a $U_q^{\textup{KD}}(\n^-)$-action, as we have seen that $U_q^{\textup{Lus}}(\n)$ and $U_q^{\textup{KD}}(\n^-)$ are dual to each other.

Now we define $U_q^{\frac{1}{2}}(\g)$-mod as the full subcategory of $U_q^{\textup{+Lus,--KD}}(\g)$-mod consisting of objects whose induced $U_q^{\textup{KD}}(\n^-)$-action factors through $U_q^{\textup{KD}}(\n^-) \twoheadrightarrow \u_q(\n^-)$. Finally, we can recover the category $U_q^{\textup{Lus}}(\g)$-mod of representations over the full (Lusztig's) quantum group as consisting of objects $(V, \vartheta)$ where $V$ is in $U_q^{\frac{1}{2}}(\g)$-mod and $\vartheta: U_q^{\textup{Lus}}(\n^-) \otimes V \to V$ an action extends the $\u_q(\n^-)$-action on $V$ along $\u_q(\n^-) \hookrightarrow U_q^{\textup{Lus}}(\n^-)$.

One can show that the forgetful functor $U_q^{\textup{Lus}}(\g)\textup{-mod} \to U_q^{\frac{1}{2}}(\g)\textup{-mod}$ is fully faithful. Therefore we have the following embeddings of categories:
\begin{equation} \label{seq_quantum_cats}
U_q^{\textup{Lus}}(\g)\textup{-mod} \hookrightarrow U_q^{\frac{1}{2}}(\g)\textup{-mod} \hookrightarrow U_q^{\textup{+Lus,--KD}}(\g)\textup{-mod}
\end{equation}

Although we define the above representation categories in terms of abelian categories, in practice we will consider the corresponding DG categories and the notations we use above will always mean DG categories. The abelian categories will be denoted by the heart of the respective DG categories. It is important to note that the embeddings in \eqref{seq_quantum_cats} still hold for DG categories.

\subsection{Kazhdan-Lusztig equivalence at negative level} \label{sec:KL_negative}
Let $\kappa'$ be a negative level and $q=\textup{exp}(\frac{\pi \sqrt{-1}}{\kappa' - \kappa_{\textup{crit}}})$. Given bilinear pairing $(\cdot,\cdot)_{\kappa'}$ on $\check{\Lambda}$, we define $b_{\kappa'}: \Lambda \otimes \Lambda \to \C^{\times}$ to be
$$b_{\kappa'}(\cdot,\cdot):= \textup{exp}\Bigl( \pi \sqrt{-1} \bigl((\cdot,\cdot)_{\kappa' -\kappa_{\textup{crit}}}\vert_{\t} \bigr)^{-1}\Bigr) \equiv q^{(\cdot, \cdot)_{\textup{st}}}.$$
Then $v_i = q^{\frac{(\alpha_i, \alpha_i)_{\textup{st}}}{2}}$ and indeed $b_{\kappa'}(\alpha_i, \alpha_j) = b_{\kappa'}(\alpha_j, \alpha_i) = v_i^{\langle \alpha_i, \check{\alpha}_j \rangle}$. Then the constructions from Section \ref{sec:quantum_groups} apply here.

The Kazhdan and Lusztig equivalence is a tensor equivalence of braided monoidal categories: $$\textup{KL}_G: (\hat{\g}_{\kappa'}\textup{-mod}^{G(\O)})^\heartsuit \tilde{\longrightarrow} (U^{\textup{Lus}}_q(\g)\textup{-mod})^\heartsuit.$$ 
Note that while the braided monoidal structure on $(U^{\textup{Lus}}_q(\g)\textup{-mod})^\heartsuit$ is given explicitly by the R-matrix and the Hopf algebra structure of $U^{\textup{Lus}}_q(\g)$, that on $(\hat{\g}_{\kappa'}\textup{-mod}^{G(\O)})^\heartsuit$ is a nontrivial construction by Kazhdan and Lusztig, inspired by Drinfeld's construction of the Drinfeld associator via the Knizhnik-Zamolodchikov equations, c.f. \cite{KazLus}.

We define the quantum Weyl module $\VV_{\lambda}$ of highest weight $\lambda$ as the image of the Weyl module $\mathbb{V}^{\kappa'}_{\lambda}$ under the (negative level) Kazhdan-Lusztig functor; i.e. $$\VV_{\lambda} := \textup{KL}_G(\mathbb{V}^{\kappa'}_{\lambda}) \in (U^{\textup{Lus}}_q(\g)\textup{-mod})^\heartsuit.$$
From \cite[Lemma~38.2]{KazLus}, our quantum Weyl module coincides with what is commonly called the Weyl module of a quantum group in the literature. Explicitly, for $\lambda \in \Lambda^+$
$$\VV_{\lambda} \cong \Ind^{U^{\textup{Lus}}_q(\g)}_{U^{\textup{Lus}}_q(\b)} \C_{\lambda}.$$

\subsection{Kazhdan-Lusztig functor at positive level} \label{sec:KL_positive}
The original Kazhdan-Lusztig functor $\textup{KL}_G$ is for negative level only. In order to define a Kazhdan-Lusztig type functor for positive level modules, we invoke the duality functor $\mathbb{D}_{G(\O)}$ defined in Section \ref{sec:affine_lie_algebra}. However, as we have seen in Lemma \ref{D_Verma}, the functor $\mathbb{D}_{G(\O)}$ does not preserve $t$-structures. Consequently our definition of the Kazhdan-Lusztig functor at positive level must involve DG categories.

Let $\kappa$ be positive (which implies that $\kappa'$ is negative) and $q=\textup{exp}(\frac{\pi \sqrt{-1}}{\kappa' - \kappa_{\textup{crit}}})$. We first derive the original Kazhdan-Lusztig equivalence to an equivalence of DG categories $\textup{KL}_G: \hat{\g}_{\kappa'}\textup{-mod}^{G(\O)} \overset{\sim}{\longrightarrow} U^{\textup{Lus}}_q(\g)\textup{-mod}$ (since the DG category $\hat{\g}_{\kappa'}\textup{-mod}^{G(\O)}$ can be recovered from its heart by ``bootstrapping"; see Section \ref{sec:affine_lie_algebra}).
Consider the contragredient duality functor for quantum group modules 
$$\mathbb{D}^q: U^{\textup{Lus}}_q(\g)\textup{-mod} \to U^{\textup{Lus}}_q(\g)\textup{-mod},$$ induced by the usual Hopf module dual at the level of abelian category. I.e. for $M \in (U^{\textup{Lus}}_q(\g)\textup{-mod})^\heartsuit$, we take the linear dual $\mathbb{D}^q(M):= \Hom_{\C}(M, \C)$ with $U^{\textup{Lus}}_q(\g)$-action twisted by the antipode.
Then we can define the Kazhdan-Lusztig functor at positive level $\kappa$
$$\textup{KL}^{\kappa}_G: \hat{\g}_{\kappa}\textup{-mod}^{G(\O)} \to U^{\textup{Lus}}_q (\g)\textup{-mod},
$$
such that the following diagram commutes:
\begin{equation} \label{diagram_positiveKL}
\xymatrix @C=6pc{
	\hat{\g}_{\kappa'}\textup{-mod}^{G(\O)} \ar[d]_{\textup{KL}_G} \ar[r]^{\mathbb{D}_{G(\O)}} & \hat{\g}_{\kappa}\textup{-mod}^{G(\O)} \ar[d]^{\textup{KL}^{\kappa}_G} \\
	U^{\textup{Lus}}_q (\g)\textup{-mod} \ar[r]_{\mathbb{D}^q} & U^{\textup{Lus}}_q (\g)\textup{-mod}
}
\end{equation}

The functor $\textup{KL}^{\kappa}_G$ is not an equivalence (for it is not $t$-exact). Nonetheless it becomes an equivalence when restricted to the full subcategory of compact objects. In fact, all arrows in \eqref{diagram_positiveKL} become equivalences when restricted to compact objects.

\section{Preparation: geometric constructions} \label{chapter:geo_constr}

\subsection{Some conventions regarding sheaves and D-modules}
For a scheme $X$, let $\O_X$ and $D_X$ denote its structure sheaf and sheaf of differential operators, respectively. 
For $X$ a scheme of finite type, the DG category of right $D_X$-modules is denoted by $\textup{D-mod}(X)$.

With the aid of higher category theory \cite{Lurie}, we extend the notion of D-modules to arbitrary prestacks: for a prestack $\mathcal{Y}$, $\textup{D-mod}(\mathcal{Y})$ is defined as the limit of $\textup{D-mod}(S)$ over the category of schemes of finite type $S$ over $\mathcal{Y}$, with structure functors given by !-pullbacks. For details see \cite{Ras-Dmod}.

On only one occasion in this paper (Section \ref{sec:proof_*thm}), we mention the \textit{ind-coherent sheaves} $\textup{IndCoh}(\mathcal{Y})$ of a prestack $\mathcal{Y}$.  The only feature we use there is the pushforward functor $f_*:\textup{IndCoh}(\mathcal{Y}) \to \textup{IndCoh}(\mathcal{Y}')$ of a morphism $f: \mathcal{Y} \to \mathcal{Y}'$. Note that under the \textit{induction functor} from $\textup{IndCoh}(\mathcal{Y})$ to $\textup{D-mod}(\mathcal{Y})$, the IndCoh pushforward corresponds to the usual de Rham (*-) pushforward of right D-modules, whenever the functors are defined.
We refer the reader to \cite{GaiRozDAG} for a full treatment of the theory of ind-coherent sheaves.

\subsection{Affine flag variety and Kashiwara-Tanisaki Localization}
Recall $I := \textup{ev}^{-1}(B)$ the Iwahori subgroup of $G(\O)$. The affine flag variety is defined as $\textup{Fl}:= G(\K)/I$, and the set of $I$-orbits of $\textup{Fl}$ is known to be indexed by the affine Weyl group $W^{\textup{aff}}:= \check{\Lambda} \rtimes W$. For $\tilde{w} \in W^{\textup{aff}}$, corresponding to the $I$-orbit $I\tilde{w}I \subset \textup{Fl}$ we denote by $j_{\tilde{w}, *}$ (resp. $j_{\tilde{w}, !}$) the costandard (resp. standard) object in the heart of the category $\textup{D-mod}(\textup{Fl})^I \cong \textup{D-mod}(I \backslash G(\K)/I)$. 

Let $\mu$ be a weight. In \cite{KT}, Kashiwara and Tanisaki constructed the category $\textup{D-mod}(\textup{Fl})^{I,\mu}$ of $\mu$-twisted $I$-equivariant right D-modules on the affine flag variety, and proved a correspondence between $\textup{D-mod}(\textup{Fl})^{I,\mu}$ and the category of $I$-equivariant modules over the affine Lie algebra at the negative level.  (See \cite[Section 2]{KT} for the detailed construction of $\textup{D-mod}(\textup{Fl})^{I,\mu}$.)

The $\mu$-twisted category admits the twisted standard and costandard objects, which will still be denoted by $j_{\tilde{w},!}$ and $j_{\tilde{w},*}$. Recall the irreducible object $j_{\tilde{w},!*}$ in the category $\textup{D-mod}(\textup{Fl})^{I,\mu}$, defined as the image of the canonical morphism $j_{\tilde{w},!} \to j_{\tilde{w},*}$. Let $\textup{D-mod}(\textup{Fl})^{I,\mu}_0$ be the full subcategory of $\textup{D-mod}(\textup{Fl})^{I,\mu}$ consisting of objects whose composition factors are isomorphic to $j_{\tilde{w},!*}$ for some $\tilde{w} \in W^{\textup{aff}}$. We state the result of Kashiwara-Tanisaki precisely as follows:

\begin{Theorem} \label{KT}
	Let $\kappa'$ be a negative level. Suppose that $\mu \in \Lambda$ satisfies $\langle \mu+\rho, \check{\alpha}_i \rangle \leq 0$ for all simple coroots $\check{\alpha}_i$.
	There is a functor between derived categories $$\mathbb{H}^{\bullet} : \textup{D-mod}(\Fl)^{I,\mu} \to \hat{\g}_{\kappa'}\textup{-mod}^I$$ with the following properties:
	\begin{enumerate}
		\item When restricted to $\textup{D-mod}(\textup{Fl})^{I,\mu}_0$, $\mathbb{H}^{n}$ is trivial for all $n \neq 0$ and $\Gamma \equiv \mathbb{H}^{0}$ is exact.
		\item $\Gamma(j_{\tilde{w},!}) \cong \mathbb{M}^{\kappa'}_{\tilde{w} \cdot \mu}$.
	\end{enumerate} 	
\end{Theorem}

Recall the dot action of $\tilde{w} \equiv \check{\lambda} w$ on a weight $\mu$ given by $$\check{\lambda} w \cdot \mu = \check{\lambda} + w(\mu+\rho)-\rho.$$
In order to be compatible with the ordinary (non-twisted) left D-modules on $G/B \hookrightarrow \Fl$ and the Beilinson-Bernstein Localization Theorem for $G/B$, we choose the twisting to be $\mu-2\rho$. Therefore, $\Gamma(j_{\check{\lambda} w,!}) \cong \mathbb{M}^{\kappa'}_{\check{\lambda} + w(\mu-\rho)-\rho}$. 

For a weight $\nu = \check{\lambda} + w(\mu-\rho)-\rho$, we define the \textit{dual affine Verma module} of highest weight $\nu$ as $\mathbb{M}^{\kappa', \vee}_{\nu} := \Gamma(j_{\check{\lambda} w,*})$. It is shown in \cite{KT} that this agrees with the contragredient dual of the affine Verma module $\mathbb{M}^{\kappa'}_{\nu}$. We recall the contragredient duality in Section \ref{sec:contragredient}. 

More generally, $\mathbb{H}^{\bullet}$ intertwines the Verdier duality on $\textup{D-mod}(\Fl)^{I,\mu}$ with the contragredient duality on $\hat{\g}_{\kappa'}\textup{-mod}^I$. We can similarly define the \textit{dual Weyl module} $\mathbb{V}^{\kappa',\vee}_{\nu}$, either algebraically via the contragredient duality or geometrically through Verdier duality by virtue of Kashiwara-Tanisaki's theorem (c.f. the discussion following the proof of Lemma \ref{VermaWeyl}.)

\subsection{Convolution action on categories} \label{sec:convolution}
Let $K \subset G(\O)$ be a compact open subgroup, and $\mathsf{C}$ be a DG category acted on strongly by $K$ on the left. This means that $\mathsf{C}$ is a left module category of $\textup{D-mod}(K)$.

Following \cite[Section 22.5]{FG}, the category $\textup{D-mod}(G(\K)/K)$ of right $K$-equivariant D-modules on the loop group acts on $\mathsf{C}^K$ by convolution, denoted by
$$- \star_K - : \textup{D-mod}(G(\K)/K) \otimes \mathsf{C}^K \to \mathsf{C}.$$
The convolution is associative in the sense that, for $K, K'$ two open compact subgroups of $G(\O)$, we have
$$(M_1 \star_K M_2) \star_{K'} X \simeq M_1 \star_K (M_2 \star_{K'} X)
$$
for $M_1$, $M_2$, $X$ objects in $\textup{D-mod}(G(\K)/K)$, $\textup{D-mod}(K \backslash G(\K)/K')$ and $\mathsf{C}^{K'}$ respectively. The identity object for the convolution action is $\delta_{K, G(\K)/K}$, the delta function D-module supported on the identity coset.

Recall the DG category $\textup{D-mod}(\Gr_G)$ of (right) D-modules on the affine Grassmannian $\Gr_G := G(\K)/G(\O)$. For a subgroup $H \subset G(\K)$, we have the $H$-equivariant category $$\textup{D-mod}(\Gr_G)^H \simeq \textup{D-mod}(H \backslash \Gr_G)$$ by considering the left $H$-action on $\textup{D-mod}(\Gr_G)$. In particular, we will consider the $N(\K)$ or $N^-(\K)$-equivariant categories as defined in \cite{semiinf}, where a special $t$-structure is constructed.
The convolution action of $\textup{D-mod}(\Gr_G)^{N^-(\K)T(\O)}$ on $\hat{\g}_{\kappa} \textup{-mod}^{G(\O)}$ is:
$$- \star_{G(\O)} - : \textup{D-mod}(\Gr_G)^{N^-(\K)T(\O)} \otimes \hat{\g}_{\kappa} \textup{-mod}^{G(\O)} \to (\hat{\g}_{\kappa} \textup{-mod})^{N^-(\K)T(\O)}.$$

We also often consider the convolution action $- \star_I - :\textup{D-mod}(\Fl)^I \otimes \hat{\g}_{\kappa}\textup{-mod}^I \to \hat{\g}_{\kappa}\textup{-mod}^I$. We prove below some identities involving convolutions that is important for our calculations of semi-infinite cohomology. Unless otherwise specified, all convolution products in the remainder of this section are with respect to $I$ and will be denoted $- \star -$.

\begin{Lemma} \label{pairingconv}
	Let $A \in \hat{\g}_{\kappa'} \textup{-mod}^I$ and $B \in \hat{\g}_{\kappa} \textup{-mod}^I$. Then we have $$\langle A, j_{\check{\lambda}, *} \star B \rangle_I = \langle j_{-\check{\lambda}, *} \star A, B \rangle_I.$$
\end{Lemma}

\begin{proof}[Proof 1]
	Recall from \eqref{eq:equiv_pairing} that the pairing $\langle-,-\rangle_I$ is characterized by $$\langle M , N \rangle_I \cong H^{\bullet}_I(\semiinf(\g(\K), M\otimes N)).$$ By \cite[Proposition 22.7.3]{FG}, we have
	\begin{equation*}
	H^{\bullet}_I(\semiinf(\g(\K), A \otimes (j_{\check{\lambda}, *} \star B)) \cong H^{\bullet}_I(\semiinf(\g(\K), (A \star j_{\check{\lambda}, *}) \otimes B),
	\end{equation*}
	where the `right convolution action' $A \star j_{\check{\lambda}, *}$ is through the left $I$-equivariance structure of $j_{\check{\lambda}, *}$ in $\textup{D-mod}(I\backslash G(\K)/I)$, opposite to the right $I$-equivariance structure we use for the left convolution action. As a consequence of this side change we have $A \star j_{\check{\lambda}, *} \cong j_{-\check{\lambda}, *} \star A$, hence we get $\langle A, j_{\check{\lambda}, *} \star B \rangle_I = \langle j_{-\check{\lambda}, *} \star A, B \rangle_I$.
\end{proof}
\begin{proof}[Proof 2 (Sketch)]
	The lemma will follow from the identity 
	\begin{equation}\label{dualconv}
	j_{-\check{\lambda}, !} \star \mathbb{D}_I A = \mathbb{D}_I( j_{-\check{\lambda}, *} \star A),
	\end{equation}
	for we will have 
	\begin{equation*}
	\begin{split}
	\langle A, j_{\check{\lambda}, *} \star B \rangle_I & = \Hom_{\hat{\g}_{\kappa}} (\mathbb{D}_I A, j_{\check{\lambda}, *} \star B) = \Hom_{\hat{\g}_{\kappa}} (j_{-\check{\lambda}, !} \star \mathbb{D}_I A, B) \\
	& = \Hom_{\hat{\g}_{\kappa}} (\mathbb{D}_I( j_{-\check{\lambda}, *} \star A), B) = \langle j_{-\check{\lambda}, *} \star A, B \rangle_I
	\end{split}
	\end{equation*}
	by combining the identity with \eqref{pairing}.
	But then, since $\mathbb{D}_I$ commutes with convolution actions (where $\mathbb{D}_I$ operates on D-modules to the same effect as taking the Verdier dual, c.f.\cite[Theorem 1.3.4]{ABBGM}), \eqref{dualconv} follows from that $j_{-\check{\lambda}, *}$ is Verdier dual to $j_{-\check{\lambda}, !}$.
\end{proof}

\begin{Lemma} \label{averaging_id}
	Let $\tilde{w} \in W^{\textup{aff}}$ and $\overline{I\tilde{w}I}$ be the closure of the orbit $I\tilde{w}I$ in $\Fl$. For $F \in \textup{D-mod}(\Gr_G)^I$, we have the following identities:
	\begin{enumerate}
		\item $\textup{Av}^{I}_!(\tilde{w} \cdot F) \cong j_{\tilde{w}, !} \star F [\dim \overline{I\tilde{w}I}],$
		\item $\textup{Av}^{I}_*(\tilde{w} \cdot F) \cong j_{\tilde{w}, *} \star F [-\dim \overline{I\tilde{w}I}].$
	\end{enumerate}
\end{Lemma}
\begin{proof}
	Unravelling the definition, we see that $\tilde{w} \cdot F \cong \delta_{\tilde{w}} \star F$, where $\delta_{\tilde{w}}$ is the delta-function D-module at the coset $\tilde{w}I \in \textup{Fl}$. Since $F$ is $I$-equivariant, the shift $\tilde{w} \cdot F$ is $\textup{Ad}_{\tilde{w}}(I)$-equivariant. Let $I^{\tilde{w}}:=I \cap \textup{Ad}_{\tilde{w}}(I)$ and $\textup{act}, \textup{pr}: I \times_{I^{\tilde{w}}} \Gr_G \to \Gr_G$ be the action and projection map, respectively.
	Then we have $$\textup{Av}^I_!(\tilde{w} \cdot F) = \textup{act}_! \circ \textup{pr}^! (\delta_{\tilde{w}} \star F) \cong \textup{act}_!(\O_I \tilde{\boxtimes} (\delta_{\tilde{w}} \star F))[\dim \overline{I\tilde{w}I}] \cong j_{\tilde{w},!} \star F [\dim \overline{I\tilde{w}I}],$$ proving (1).
	
	For (2), we apply the Verdier duality on D-modules to (1). Since the Verdier duality commutes with convolution product, we get
	$$\textup{Av}^{I}_*(\tilde{w} \cdot \mathbb{D}F) \cong j_{\tilde{w}, *} \star \mathbb{D}F [-\dim \overline{I\tilde{w}I}].
	$$
	Since $\mathbb{D}$ is an equivalence, (2) is proven.
\end{proof}

\subsection{Spherical category}
We define the spherical category as 
$$\textup{Sph}_G := \textup{D-mod}(\Gr_G)^{G(\O)},$$
the DG category of (left) $G(\O)$-equivariant D-modules on $\Gr_G$, which is a categorical analogue of the spherical Hecke algebra in number theory. The $G(\O)$-orbits of $\Gr_G$ are parametrized by the dominant coweights of $G$. For each dominant coweight $\check{\lambda} \in \check{\Lambda}^+$, the corresponding $G(\O)$-orbit is $G(\O)t^{\check{\lambda}}G(\O) =: \Gr_G^{\check{\lambda}}$. Let $\textup{IC}_{\check{\lambda}} \in (\textup{Sph}_G)^{\heartsuit}$ be the IC D-module supported on the closure $\overline{\Gr_G^{\check{\lambda}}}$. It is known that $(\textup{Sph}_G)^{\heartsuit}$ is semisimple with simple objects given by these IC D-modules.

Recall from \cite{MV} that the spherical category is equipped with a convolution product $- \star_{G(\O)} -$, which makes $(\textup{Sph}_G)^\heartsuit$ a symmetric monoidal category. In fact, this coincides with the convolution defined in Section \ref{sec:convolution} by taking $\mathsf{C} = \textup{D-mod}(\Gr_G)$ and $K=G(\O)$.

Denote by $\textup{Sat}: \textup{Rep}(\check{G})^\heartsuit \to (\textup{Sph}_G)^\heartsuit$ the geometric Satake equivalence, a tensor equivalence of symmetric monoidal categories. For $\check{\lambda} \in \check{\Lambda}^+$, the functor $\textup{Sat}$ sends the irreducible representation $V_{\check{\lambda}}$ of $\check{G}$ to the IC D-module $\textup{IC}_{\check{\lambda}}$.

We define the semi-infinite orbits $S_{\check{\mu}}$ as the $N(\K)$-orbit $N(\K)t^{\check{\mu}}G(\O)$ inside $\Gr_G$.
The opposite semi-infinite orbit $T_{\check{\nu}}$ is defined to be $N^-(\K)t^{\check{\nu}}G(\O)$. It is known that the semi-infinite orbits and the opposite semi-infinite orbits are both parametrized by the coweight lattice $\check{\Lambda}$ of $G$.

The theory of weight functors developed in \cite{MV} enables us to compute cohomology of objects in the spherical category by representation theory. In particular, we will compute the !-stalk of $\textup{IC}_{\check{\lambda}}$ at the point $t^{w_0(\check{\lambda})}G(\O) \in \Gr_G$. Denote the inclusions
\begin{equation*}
\xymatrix @R=1pc @M=0.5pc{ t^{w_0(\check{\lambda})}G(\O) \ar@{^{(}->}[rr]^{\iota} \ar@{^{(}->}[dr]_{k} & & \Gr_G \\
	& S_{w_0(\check{\lambda})} \ar@{^{(}->}[ur]_{s} &
}
\end{equation*}
Then the !-stalk of $\textup{IC}_{\check{\lambda}}$ at $t^{w_0(\check{\lambda})}G(\O)$ is 
\begin{equation*}
\iota^!\, \textup{IC}_{\check{\lambda}} \cong k^! s^*\, \textup{IC}_{\check{\lambda}} [2\langle \rho, w_0(\check{\lambda})-\check{\lambda} \rangle] \cong H^\bullet_c(S_{w_0(\check{\lambda})}, \textup{IC}_{\check{\lambda}})[2\langle \rho, w_0(\check{\lambda})-\check{\lambda} \rangle].
\end{equation*}
But then the cohomology $H^\bullet_c(S_{w_0(\check{\lambda})}, \textup{IC}_{\check{\lambda}})$ vanishes except at degree $\langle 2\rho, w_0(\check{\lambda})\rangle$, and the non-vanishing part is precisely the weight functor that computes the weight multiplicity of $V_{\check{\lambda}}$ at weight $w_0(\check{\lambda})$. The representation theory tells us that it is one-dimensional. We conclude
\begin{equation} \label{!stalk_IC}
\iota^! \,\textup{IC}_{\check{\lambda}} \cong \C[-\langle 2\rho, \check{\lambda}\rangle] \cong \C[\langle 2\rho, w_0(\check{\lambda})\rangle].
\end{equation}

For arbitrary level $\kappa$, we define an action of $\textup{Rep}(\check{G})$ on $\hat{\g}_{\kappa}\textup{-mod}^{G(\O)}$ by
$$V, M \mapsto \textup{Sat}(V) \star_{G(\O)} M.
$$
According to \cite[Theorem 1.3.4]{ABBGM}, we have
\begin{equation} \label{convolution_Fr}
\textup{KL}_G(\textup{Sat}(V) \star_{G(\O)} M) \cong \textup{Fr}_q(V) \otimes \textup{KL}_G(M)
\end{equation}
for $V \in \textup{Rep}(\check{G})$ and $M \in \hat{\g}_{\kappa'}\textup{-mod}^{G(\O)}$ where $\kappa'$ is negative.

\section{Wakimoto modules} \label{chapter:wakimoto}

\subsection{First construction} \label{firstcon_Wakimoto}
The Wakimoto modules are a class of representations of affine Kac-Moody algebras, originally introduced by M. Wakimoto \cite{wakimoto1986} for $\hat{\mathfrak{sl}_2}$ and generalized to arbitrary types by B. Feigin and E. Frenkel \cite{feigin1988}. An algebraic construction of Wakimoto modules can be found in \cite{Fre}.

In \cite[Part III]{FG}, a geometric construction of Wakimoto modules via the language of chiral algebras is given. Inspired by the localization theorem of Beilinson and Bernstein \cite{BB1981}, the construction can be seen as an infinite-dimensional analog of taking sections of twisted D-modules on the big Schubert cell in $G/B$. Given a level $\kappa$, a Weyl group element $w$ and a weight $\lambda$, it produces a $\hat{\g}_{\kappa}$-module $\mathbb{W}^{\kappa,w}_{\lambda}$, called the type $w$ Wakimoto module of highest weight $\lambda$ at level $\kappa$. We shall not repeat the construction here.

A crucial property of Wakimoto modules following from this line of construction is:
\begin{Proposition}[\cite{FG} Proposition 12.5.1]  \label{Wakimoto_wt_shift}
	For a dominant coweight $\check{\lambda} \in \check{\Lambda}^+$, we have
	$$j_{\check{\lambda},!} \star_I \mathbb{W}^{\kappa, 1}_{\mu} \cong \mathbb{W}^{\kappa, 1}_{\mu + \check{\lambda}} \quad \textup{and} \quad j_{\check{\lambda},*} \star_I \mathbb{W}^{\kappa, w_0}_{\mu} \cong \mathbb{W}^{\kappa, w_0}_{\mu + \check{\lambda}}.$$
\end{Proposition}

\subsection{Digression: Modules over affine Kac-Moody algebras and contragredient duality} \label{sec:contragredient}
In order to compare Verma modules and Wakimoto modules algebraically, we need to introduce the notion of character of a module over affine Lie algebras. However, characters are well-defined only when weight spaces are finite-dimensional, which prompts us to introduce the action of the degree operator $t\partial_t$ and affine Kac-Moody algebras $\hat{\g}\rtimes \C t \partial_t$.

Let $\hat{\g}:= \hat{\g}_1$. Consider the affine Kac-Moody algebra $\hat{\g} \rtimes \C t \partial_t$, where the degree operator $t \partial_t$ acts on $\hat{\g}$ by
$$t \partial_t (g) := t \dfrac{dg}{dt}\quad \textrm{for} \, g \in \g(K) \quad \textrm{and} \quad t \partial_t (\textbf{1}) := 0.$$

Given any $\g$-module $V$, if we let $t\partial_t$ act on $V$ by $0$ and $\textbf{1}$ act by $\kappa$, then the induction makes $$V^{\kappa}:= \Ind^{\hat{\g} \rtimes \C t \partial_t}_{\g (\O) \oplus \C \textbf{1} \oplus \C t \partial_t}V$$
a module over the affine Kac-Moody algebra. We will consider the category of affine Kac-Moody modules where the central element $\textbf{1}$ acts by $\kappa$, and we will call its objects $\hat{\g}_{\kappa} \rtimes \C t \partial_t$-modules or modules over $\hat{\g}_{\kappa} \rtimes \C t \partial_t$. Clearly, a $\hat{\g}_{\kappa} \rtimes \C t \partial_t$-module is equivalent to a $\hat{\g}_{\kappa}$-module with the same $t \partial_t$-action.

Clearly we can define the Verma module $\mathbb{M}^{\kappa}_{\lambda}$ and Weyl module $\mathbb{V}^{\kappa}_{\lambda}$ over $\hat{\g}_{\kappa} \rtimes \C t \partial_t$ by the same induction procedure. To make the Wakimoto modules $\W^{\kappa,w}_{\lambda}$ an affine Kac-Moody module, we let the operator $t\partial_t$ act by loop rotation. Namely, the action is the unique compatible action induced from the requirement that the vacuum vector of weight $\lambda$ is annihilated by $t\partial_t$.

A weight of an affine Kac-Moody algebra is a tuple $(n, \mu, v) \in (\C t \partial_t \oplus \t \oplus \C\textbf{1})^*$, with the natural pairing given by $(n, \mu, v) \cdot (p t \partial_t, h, a\textbf{1}) = np+\mu(h)+va$. From the structure theory of affine Kac-Moody algebras, the set of roots of $\hat{\g} \rtimes \C t \partial_t$ is
$$\hat{R} = \{ (n, \alpha, 0): n \in \Z, \alpha \in R \sqcup 0 \}-\{(0, 0, 0)\},$$
where $R$ is the root system of $\g$. The roots of the form $(n, 0, 0)$ are called imaginary roots, each of which has multiplicity equal to $\dim \t$. All the other roots are called real, with multiplicity one. The set of positive roots is $$\hat{R}^+ =\{(0, \alpha, 0): \alpha \in R^+\} \sqcup \{(n, \alpha, 0): n>0, \alpha \in R \sqcup 0\}.$$

We denote by $M(\hat{\mu})$ the $\hat{\mu}$-weight space of an affine Kac-Moody module $M$.
Note that, due to the grading given by the action of $t\partial_t$, all weight spaces of $V^{\kappa}$ for any $V \in \g\textup{-mod}$ are finite-dimensional.
For an affine Kac-Moody module $M$ with finite-dimensional weight spaces, we define the character of $M$ to be the formal sum
$$ch\,M := \sum_{\hat{\mu} \in (\C t \partial_t \oplus \t \oplus \C\textbf{1})^*} \dim M(\hat{\mu})\, e^{\hat{\mu}}.
$$

The Cartan involution $\tau$ on $\g$ is a linear involution which sends the Chevalley generators as follows:
$$\tau(e_i) = -f_i, \quad \tau(f_i) = -e_i, \quad \tau(h_i) = -h_i.$$
For a $\g$-module $V$, let $V^{\vee}:= \oplus_{\mu \in \t^*} V(\mu)^*$ be the usual contragredient dual, with its $\g$-action defined by $x \cdot f(v):=f(-\tau(x)\cdot v)$ for $f \in V^*$ and $x \in \g$.

We extend $\tau$ to an involution $\hat{\tau}$ on $\hat{\g} \rtimes \C t \partial_t$ by setting $\hat{\tau}(f_{\theta} \otimes t) := -e_{\theta} \otimes t^{-1}$, $\hat{\tau}(\textbf{1}) := -\textbf{1}$ and $\hat{\tau}(t \partial_t) = -t \partial_t$. Here $\theta$ denotes the longest root of $\g$.

Given any affine Kac-Moody modules $M$ with finite-dimensional weight spaces, we similarly define its contragredient dual $M^{\vee}$ as the restricted dual space $$\bigoplus_{\hat{\mu} \in (\C t \partial_t \oplus \t \oplus \C1)^*} M(\hat{\mu})^*$$ with the $\hat{\g} \rtimes \C t \partial_t$-action given by the same formula with $\tau$ replaced by $\hat{\tau}$.

The following basic properties are evident by definition:

\begin{Proposition} \label{th:dual}
	$\,$
	\begin{enumerate}
		\item For an affine Kac-Moody module $M$ with finite-dimensional weight spaces, $(M^{\vee})^{\vee} = M$. 
		
		\item Taking contragredient dual of a representation preserves its character.
	\end{enumerate}
\end{Proposition}

\subsection{Camparing Wakimoto and Verma modules at negative level}
When the level is negative, we show that the Wakimoto module of type $w_0$ is isomorphic to the dual Verma module of the same highest weight $\lambda$ if $\lambda$ is sufficiently dominant (Theorem \ref{th:WakiVer}).

Let us fix a negative level $\kappa'$ from now on. The character of $\mathbb{M}^{\kappa'}_{\lambda}$ is by its definition $$ch\, \mathbb{M}^{\kappa'}_{\lambda} = e^{(0, \lambda, \kappa')} \biggl( \prod_{\hat{\alpha}\, \in \,{\hat{R}}^+}\left(1-e^{-\hat{\alpha}}\right)^{mult(\hat{\alpha})} \biggr)^{-1}.$$

The following lemma describes the character of the Wakimoto module $\W^{\kappa',w_0}_{\lambda}$. In particular, we see that the characters of Verma module  $\mathbb{M}^{\kappa'}_{\lambda}$ and Wakimoto module $\W^{\kappa',w_0}_{\lambda}$ are identical.

\begin{Lemma} \label{th:char}
	$$\textrm{ch}\, \W^{\kappa',w_0}_{\lambda} = e^{(0, \lambda, \kappa')} \biggl( \prod_{\hat{\alpha}\, \in \,{\hat{R}}^+}\left(1-e^{-\hat{\alpha}}\right)^{mult(\hat{\alpha})} \biggr)^{-1}.$$
\end{Lemma}

\begin{proof}
	As a $\C t \partial_t \oplus \t \oplus \C\textbf{1}$-module, $\W^{\kappa',w_0}_{\lambda}$ is equal to 
	\begin{equation} \label{character}
	\C[x^*_{\alpha, n}]_{\alpha \in R^+, n \leq 0} \otimes \C[x_{\alpha, m}]_{\alpha \in R^+, m<0} \otimes \C[y_{i, l}]_{l<0, i=1, \dotsc, \dim\t},
	\end{equation}
	where the weights of $x^*_{\alpha, n}$, $x_{\alpha, m}$ and $y_{i, l}$ are $(n, -\alpha, 0)$, $(m, \alpha, 0)$ and $(l, 0, 0)$, respectively, and the vector $1 \otimes 1 \otimes 1$ has weight $(0, \lambda, \kappa')$.
	We deduce this by \cite[formula (11.6)]{FG} and see that $\C[y_{i, l}]_{l<0, i=1, \dotsc, \dim\t}$ corresponds to the Fock module $\pi^{-\kappa-\textup{shift}}_{w_0(\lambda)}$ over the Heisenberg algebra, the tensor factor $\C[x^*_{\alpha, n}]_{\alpha \in R^+, n \leq 0}$ corresponds to $\textup{Fun}(N[[t]])$, and $\C[x_{\alpha, m}]_{\alpha \in R^+, m<0}$ arises from the induction of $\textup{Fun}(N[[t]])$ to a chiral module over chiral differential operator $\mathfrak{D}^{\textup{ch}}(N)$.
	The character formula follows immediately from \eqref{character}.
\end{proof}

We recall the Kac-Kazhdan Theorem \cite{KacKaz} on possible singular weights appearing in a Verma module over $\hat{\g}_{\kappa} \rtimes \C t \partial_t$ (where $\kappa$ is arbitrary):

\begin{Theorem} [Kac-Kazhdan]
	Let $\hat{\mu} = (n, \mu, \kappa)$ be a singular weight of $\mathbb{M}^{\kappa}_{\lambda}$, namely, $\hat{\mu}$ is a highest weight of some subquotient of $M^{\kappa}_{\lambda}$. Then the following condition holds:
	There exist a sequence of weights $(0, \lambda, \kappa) = \hat{\lambda} \equiv \hat{\mu}_1, \hat{\mu}_2, \dotsc , \hat{\mu}_n \equiv \hat{\mu}$ and a sequence of positive roots $\hat{\alpha}_k \in \hat{R}^+$, $k = 1, 2, \dotsc , n$, such that for each $k$, there is $b_k \in \Z^{>0}$ satisfying
	$$\hat{\mu}_{k+1} = \hat{\mu}_k - b_k \cdot \hat{\alpha}_k$$
	and
	$$b_k \cdot (\hat{\alpha}_k , \hat{\alpha}_k) = 2 \cdot (\hat{\alpha}_k , \hat{\mu}_k + (0, \rho, h^{\vee})).$$
	Here $(\cdot, \cdot)$ is the standard invariant bilinear form on the weights of $\hat{\g} \rtimes \C t \partial_t$.
\end{Theorem}

The main result of this section is the following:
\begin{Theorem} \label{th:WakiVer}
	$\,$
	\begin{enumerate}
		\item Let $\kappa'$ be negative and rational. Suppose that $\lambda \in \Lambda^+$ is sufficiently dominant. Then $\W^{\kappa',w_0}_{\lambda}$ is isomorphic to $(\mathbb{M}^{\kappa'}_{\lambda})^{\vee}$ as $\hat{\g}_{\kappa'}$-modules.
		\item Let $\kappa'$ be irrational and $\lambda$ be integral. Then $\W^{\kappa',w_0}_{\lambda}$ is isomorphic to $(\mathbb{M}^{\kappa'}_{\lambda})^{\vee}$ as $\hat{\g}_{\kappa'}$-modules.
	\end{enumerate}
\end{Theorem}

\begin{proof} [Proof of 2.]
	We will prove $\mathbb{M}^{\kappa'}_{\lambda} \cong (\W^{\kappa',w_0}_{\lambda})^{\vee}$.
	
	First of all, since $(\W^{\kappa',w_0}_{\lambda})^{\vee}$ has highest weight the same as that of $\mathbb{M}^{\kappa'}_{\lambda}$, the universal property of Verma module implies the existence of a canonical morphism $$\Phi : \mathbb{M}^{\kappa'}_{\lambda} \to (\W^{\kappa',w_0}_{\lambda})^{\vee}.$$
	Using Lemma \ref{th:char} and Proposition \ref{th:dual}, we see that $\mathbb{M}^{\kappa'}_{\lambda}$ and $(\W^{\kappa',w_0}_{\lambda})^{\vee}$ have the same character. Hence $\Phi$ is injective if and only if it is surjective.
	
	Suppose that $\Phi$ is not injective. Then we can pick a highest weight vector $u \in \Ker(\Phi)$ of weight $\hat{\mu}$. By the equality of characters, there exists $v \in (\W^{\kappa',w_0}_{\lambda})^{\vee} / \textrm{Im}(\Phi)$ of the same weight $\hat{\mu}$. Now we claim that $v$ cannot lie in $(\n^-[t^{-1}] \oplus t^{-1}\b[t^{-1}]) (\W^{\kappa',w_0}_{\lambda})^{\vee}$. Indeed, if $v \in (\n^-[t^{-1}] \oplus t^{-1}\b[t^{-1}]) (\W^{\kappa',w_0}_{\lambda})^{\vee}$, then we can find a vector $v' \in (\W^{\kappa',w_0}_{\lambda})^{\vee}$ of weight higher than $\hat{\mu}$ such that $x \cdot v' = v$ for some $x \in \hat{\g}$. By the assumption on $\hat{\mu}$, we obtain a vector $u' \in M^{\kappa'}_{\lambda}$ with $\Phi (u')=v'$, but then $v = x \cdot \Phi(u') = \Phi (x \cdot u') \in \textrm{Im}(\Phi)$ is a contradiction.
	
	Therefore, the vector $v$ projects nontrivially onto the coinvariants $$(\W^{\kappa',w_0, \vee}_{\lambda})_{\n^-[t^{-1}] \oplus t^{-1}\b[t^{-1}]},$$ and in particular, as $\C t \partial_t \oplus \t \oplus \C1$-modules,
	\begin{equation*}
	(\W^{\kappa',w_0,\vee}_{\lambda})_{\n^-[t^{-1}] \oplus t^{-1}\b[t^{-1}]}
	\twoheadleftarrow (\W^{\kappa',w_0,\vee}_{\lambda})_{\n^-[t^{-1}] \oplus t^{-1}\t[t^{-1}]}
	\cong \C[x_{\alpha, m}]_{\alpha \in R^+, m<0}
	\end{equation*}
	(for the notation $x_{\alpha,m}$, see the proof of Lemma \ref{th:char}).
	We conclude that the weight $\hat{\mu}$ must be of the form 
	\begin{equation} \label{eq:mu}
	\hat{\mu} = (-n, \lambda + \beta, \kappa')
	\end{equation}
	for $n \in \Z^{>0}, \beta \in \textrm{Span}^+(R^+)$.
	
	On the other hand, since $\hat{\mu}$ is a highest weight of a submodule of $\mathbb{M}^{\kappa'}_{\lambda}$, there exist a sequence of weights $(0, \lambda, \kappa') = \hat{\lambda} \equiv \hat{\mu}_1, \hat{\mu}_2, \dotsc , \hat{\mu}_n \equiv \hat{\mu}$ and a sequence of positive roots $\hat{\alpha}_k, k = 1, 2, \dotsc , n$, satisfying the conditions in the Kac-Kazhdan Theorem. For each $k = 1, 2, \dotsc , n$, either $\hat{\alpha}_k$ is real or it is imaginary. We write $\hat{\mu}_k = (n_k, \mu_k, \kappa')$, $\hat{\alpha}_k = (m_k, \alpha_k, 0), m_k \geq 0, \alpha_k \in R$ if $\hat{\alpha}_k$ is real, and $\hat{\alpha}_k = (m_k, 0, 0), m_k > 0$ if $\hat{\alpha}_k$ is imaginary.
	
	Suppose that $\hat{\alpha}_k$ is real with $m_k$ nonzero. Then
	\begin{equation} \label{eq:b_k}
	\begin{split}
	b_k & = p \cdot (\hat{\alpha}_k, \hat{\mu}_k + (0, \rho, h^{\vee})) \\ &
	= p \cdot ((m_k, \alpha_k, 0),(n_k, \mu_k+\rho, \kappa' + h^{\vee})) \\ &
	= p \cdot (\kappa' + h^{\vee}) m_k + p \cdot (\mu_k + \rho, \alpha_k),
	\end{split}
	\end{equation} 
	where $p$ is some nonzero positive rational number. From our assumption that $\kappa'$ is irrational and $\lambda$ is integral, we get an irrational number $b_k$, which contradicts the condition in the Kac-Kazhdan Theorem.
	
	Therefore, $\hat{\alpha}_k$ has to be either real with $m_k = 0$ or imaginary for each $k$. This implies $\mu_{k+1} = \mu_k - \alpha_k$ for $\alpha_k \in R^+ \sqcup 0$, and so 
	\begin{equation} \label{eq:mu_wrong}
	\hat{\mu} = (-n, \lambda - \beta, \kappa'),
	\end{equation}
	where $\beta \in \textrm{Span}^+(R^+) \sqcup 0$. But then this is a contradiction to the form of $\hat{\mu}$ we obtained in (\ref{eq:mu}).
\end{proof}
\begin{proof} [Proof of 1.]
	The same argument as in the previous proof leads to formula \eqref{eq:b_k} for $b_k$ when $\hat{\alpha}_k$ is real. Now since $\kappa'$ is assumed rationally negative and $\lambda$ is sufficiently dominant, $b_k$ can possibly be positive only when $\alpha_k$ is a positive root of $\g$. Therefore either $\hat{\alpha}_k$ is imaginary, or $\hat{\alpha}_k$ is real and $\mu_{k+1} = \mu_k - b_k \alpha_k$ for $\alpha_k \in R^+$. Again we arrive at the expression $$\hat{\mu} = (-n, \lambda - \beta, \kappa')$$ where $\beta \in \textrm{Span}^+(R^+) \sqcup 0$, contradicting \eqref{eq:mu}.
\end{proof}

\subsection{Second construction, via convolution} \label{secondcon_Wakimoto}
Following \cite{quantumsemiinf}, the second approach to construct the Wakimoto modules incorporates into its definition the feature that Wakimoto modules are stable under convolution with $j_{\check{\lambda},!}$ or $j_{\check{\lambda},*}$ (Proposition \ref{Wakimoto_wt_shift}).
It is sufficient for our purpose to define two types of Wakimoto modules, $\mathbb{W}^{\kappa', *}_{\lambda}$ and $\mathbb{W}^{\kappa', w_0}_{\lambda}$, at a negative level $\kappa'$.

When $\lambda$ is sufficiently dominant, set $\mathbb{W}^{\kappa', *}_{\lambda} := \mathbb{M}^{\kappa'}_{\lambda}$. For general $\lambda$, write $\lambda = \lambda_1 - \check{\mu}$, where $\lambda_1$ is sufficiently dominant and $\check{\mu} \in \check{\Lambda}^+$, and define
\begin{equation*}
\mathbb{W}^{\kappa', *}_{\lambda} := j_{-\check{\mu}, *} \star_I \mathbb{W}^{\kappa', *}_{\lambda_1}.
\end{equation*}
Note that this is well-defined as $\mathbb{M}^{\kappa'}_{\check{\mu}+\lambda} \cong j_{\check{\mu},!} \star_I \mathbb{M}^{\kappa'}_{\lambda}$ for $\lambda$ sufficiently dominant and any $\check{\mu} \in \check{\Lambda}^+$ by Kashiwara-Tanisaki localization, and $j_{-\check{\mu},*} \star j_{\check{\mu},!} \star - \simeq \textup{Id}$ for $\check{\mu} \in \check{\Lambda}^+$. Then by definition
\begin{equation} \label{*Wakimoto}
\mathbb{W}^{\kappa', *}_{\lambda - \check{\mu}} \cong j_{-\check{\mu}, *} \star_I \mathbb{W}^{\kappa', *}_{\lambda}
\end{equation}
holds for all $\lambda \in \Lambda$ and $\check{\mu} \in \check{\Lambda}^+$. If $\lambda$ is sufficiently anti-dominant, it can be shown that $\mathbb{W}^{\kappa', *}_{\lambda} \cong \mathbb{M}^{\kappa', \vee}_{\lambda}$.

We define $\mathbb{W}^{\kappa', w_0}_{\lambda}$ analogously. Put $\mathbb{W}^{\kappa', w_0}_{\lambda} := \mathbb{M}^{\kappa', \vee}_{\lambda}$ when $\lambda$ is sufficiently dominant. For general $\lambda$, again write $\lambda = \lambda_1 - \check{\mu}$, where $\lambda_1$ is sufficiently dominant and $\check{\mu} \in \check{\Lambda}^+$, and define
\begin{equation*}
\mathbb{W}^{\kappa', w_0}_{\lambda} := j_{-\check{\mu}, !} \star_I \mathbb{W}^{\kappa', w_0}_{\lambda_1}.
\end{equation*}

We have seen that the type $w_0$ Wakimoto modules introduced in Section \ref{firstcon_Wakimoto} satisfies $\mathbb{W}^{\kappa',w_0}_{\lambda} \cong j_{-\check{\mu},!} \star_I \mathbb{W}^{\kappa',w_0}_{\lambda + \check{\mu}}$ for dominant $\check{\mu}$ (Proposition \ref{Wakimoto_wt_shift}) and $\mathbb{W}^{\kappa', w_0}_{\lambda} \cong \mathbb{M}^{\kappa', \vee}_{\lambda}$ for sufficiently dominant $\lambda$ (Theorem \ref{th:WakiVer}). Consequently $\mathbb{W}^{\kappa',w_0}_{\lambda}$ defined here using convolution agrees with the type $w_0$ Wakimoto module in Section \ref{firstcon_Wakimoto}. One can similarly show that $\mathbb{W}^{\kappa',*}_{\lambda}$ is identified with $\mathbb{W}^{\kappa',1}_{\lambda}$.

From this construction, it is clear that both $\mathbb{W}^{\kappa', *}_{\lambda}$ and $\mathbb{W}^{\kappa', w_0}_{\lambda}$ lie in the category $\hat{\g}_{\kappa'}\textup{-mod}^I$.

\subsection{Relations to semi-infinite cohomology} \label{sec:rel_to_semiinf}
As its construction involves semi-infinite cohomology, it is not surprising that Wakimoto modules are closely related to semi-infinite calculus. Below we present two formulas for computing semi-infinite cohomology, one at negative level and the other at positive level.

The first formula is a result from \cite{FG}:
\begin{Proposition} [\cite{FG} Proposition 12.4.1] \label{th:WakiBRST}
	$\semiinf(\n(\K), \mathbb{W}^{\kappa', w_0}_{\lambda})$ is isomorphic to the Fock module $\pi^{\kappa' + \textup{shift}}_{\lambda}$ (placed at homological degree 0) as complexes of modules over the Heisenberg algbera $\hat{\t}_{\kappa' + \textup{shift}}$.
\end{Proposition}

The second formula concerns the semi-infinite cohomology of modules at a positive level $\kappa$. Let $M \in \hat{\g}_{\kappa}\textup{-mod}^I$. We claim
\begin{Proposition} \label{wakimoto_pairing}
	$\langle \mathbb{W}^{\kappa', *}_{-\mu-2\rho}[\dim(G/B)] \, ,\, M \rangle_I \cong \semiinf(\n(\K), M)^{\mu}.$
\end{Proposition}
\begin{proof}
	The $\mu$-component of the semi-infinite complex can be computed by
	$$\semiinf(\n(\K), M)^{\mu} \cong \underset{\check{\lambda} \in \check{\Lambda}}{\textup{colim}}\, \textup{C}^{\bullet}(\textup{Lie}(I^0), \textup{Av}^{I}_*( t^{\check{\lambda}} \cdot M))^{\check{\lambda}+\mu}[\langle \check{\lambda}, 2\rho \rangle],
	$$ as in \cite[Proposition 1.2.3]{quantumsemiinf}. On the other hand, by Lemma \ref{pairingconv} and Lemma \ref{D_Verma} we have
	\begin{equation*}
	\begin{split}
	&\quad \langle \mathbb{W}^{\kappa', *}_{-\mu-2\rho}[\dim(G/B)] \, ,\, M \rangle_I \cong \underset{\check{\lambda} \in \check{\Lambda}}{\textup{colim}}\,\langle j_{-\check{\lambda}, *} \star_I \mathbb{M}^{\kappa'}_{-\mu-2\rho + \check{\lambda}}[\dim(G/B)] \, ,\, M \rangle_I \\
	& \cong \underset{\check{\lambda} \in \check{\Lambda}}{\textup{colim}}\,\langle \mathbb{M}^{\kappa'}_{-\mu-2\rho + \check{\lambda}}[\dim(G/B)] \, ,\, j_{\check{\lambda}, *} \star_I M \rangle_I \cong \underset{\check{\lambda} \in \check{\Lambda}}{\textup{colim}}\,\Hom_{\hat{\g}_{\kappa}\textup{-mod}^I}(\mathbb{M}^{\kappa}_{\mu + \check{\lambda}},\, j_{\check{\lambda}, *} \star_I M).
	\end{split}
	\end{equation*}
	
	Now, Lemma \ref{averaging_id} gives
	\begin{equation*}
	\begin{split}
	\Hom_{\hat{\g}_{\kappa}\textup{-mod}^I}(\mathbb{M}^{\kappa}_{\mu + \check{\lambda}},\, j_{\check{\lambda}, *} \star_I M)
	\cong  \Hom_{\hat{\g}_{\kappa}\textup{-mod}^I}(\mathbb{M}^{\kappa}_{\mu + \check{\lambda}},\, \textup{Av}^I_*(t^{\check{\lambda}} \cdot M))[\langle \check{\lambda}, 2\rho \rangle] \\
	\cong \textup{C}^{\bullet}(\textup{Lie}(I^0), \textup{Av}^{I}_*( t^{\check{\lambda}} \cdot M))^{\check{\lambda}+\mu}[\langle \check{\lambda}, 2\rho \rangle].
	\end{split}
	\end{equation*}
	The proposition follows.
\end{proof}

\section{Semi-infinite cohomology vs quantum group cohomology} \label{chapter:semiinf_pos}

This section presents the main result of this paper, Theorem \ref{*thm} which compares the semi-infinite cohomology at positive level and the quantum group cohomology. For the proof of the theorem, we introduce the generalized semi-infinite cohomology functors, with the usual semi-infinite cohomology being a special case.

\subsection{Weyl modules revisited}
In this section we prove two technical lemmas concerning Weyl modules.

\begin{Lemma} \label{VermaWeyl}
	Let $\lambda$ be a dominant integral weight.
	We have isomorphic $\hat{\g}_{\kappa'}$-modules $$\textup{Av}_!^{G(\O)/I} \mathbb{M}^{\kappa'}_{\lambda} \cong \mathbb{V}^{\kappa'}_{\lambda}.$$
\end{Lemma}

\begin{proof}
	By viewing the induction $\Ind^{\hat{\g}_{\kappa'}}_{\g(\O)\oplus \C \textbf{1}}$ as the left adjoint functor to the restriction functor, it is easy to verify that $$\textup{Av}^{G(\O)/I}_! \circ \Ind^{\hat{\g}_{\kappa'}}_{\g(\O)\oplus \C \textbf{1}} \circ \Ind^{\g}_{\b} \simeq \Ind^{\hat{\g}_{\kappa'}}_{\g(\O)\oplus \C \textbf{1}} \circ \textup{Av}^{G/B}_! \circ \Ind^{\g}_{\b}$$ as functors from $\b \textup{-mod}$ to $\hat{\g}_{\kappa'}\textup{-mod}^I$.
	
	We have $\textup{Av}_!^{G(\O)/I} \mathbb{M}^{\kappa'}_{\lambda} \cong \Ind^{\hat{\g}_{\kappa'}}_{\g(\O)\oplus \C \textbf{1}} \circ \textup{Av}^{G/B}_! M_{\lambda}$ as $M_{\lambda} \cong \Ind^{\g}_{\b}\C_{\lambda}$. Thus it remains to prove the isomorphism as $\g$-modules $$\textup{Av}^{G/B}_! M_{\lambda} \cong V_{\lambda}.$$
	
	By the Beilinson-Bernstein localization combined with the translation functor (c.f. \cite{GaiRepTh}), we have $$M^{\vee}_{\lambda} \cong \Gamma(G/B, \textup{Dist}(Bw_0B) \otimes_{\O_{G/B}} \mathcal{L}_{-w_0(\lambda)}),$$ where $\textup{Dist}(Bw_0B)$ is the *-pushforward of the sheaf of regular functions on $Bw_0B \subset G/B$, considered as a left D-module on $G/B$, and $\mathcal{L}_{-w_0(\lambda)}$ is the line bundle induced from $\O_G \otimes_{B} \C_{w_0(\lambda)}$.
	It is easy to see that 
	\begin{equation} \label{*av}
	\textup{Av}^{G/B}_*(\textup{Dist}(Bw_0B) \otimes_{\O_{G/B}} \mathcal{L}_{-w_0(\lambda)}) \cong \O_{G/B} \otimes_{\O_{G/B}} \mathcal{L}_{-w_0(\lambda)} \cong \mathcal{L}_{-w_0(\lambda)}.
	\end{equation}
	The global section of the line bundle $\mathcal{L}_{-w_0(\lambda)}$ (regarded as a twisted D-module) matches the irreducible $\g$-module $V_{\lambda}$ by the Borel-Weil-Bott Theorem.
	Finally, we apply the Verdier dual on both sides of \eqref{*av} and take the global sections to conclude $\textup{Av}^{G/B}_! M_{\lambda} \cong V_{\lambda}$, as the Verdier duality on $\textup{D-mod}(G/B)$ corresponds to the contragredient duality on $\g$-modules.
\end{proof}

Interpreting Lemma \ref{VermaWeyl} by Kashiwara-Tanisaki's theorem (Theorem \ref{KT}) , we see that the Weyl module $\mathbb{V}^{\kappa'}_{\lambda}$ can be constructed geometrically as the global section of the $\mu$-twisted D-module $\textup{Av}^{G(\O)/I}_! j_{\tilde{w},!}$ for some $\tilde{w} \in W^{\textup{aff}}$ and $\mu \in \Lambda$. Note that, $\textup{Av}^{G(\O)/I}_! j_{\tilde{w},!}$ is precisely the !-pushforward of $\O_{G(\O)\tilde{w}I}$, the sheaf of regular functions on the $G(\O)$-orbit $G(\O)\tilde{w}I \subset \Fl$. This enables us to define the dual Weyl module as $$\mathbb{V}^{\kappa', \vee}_{\lambda} := \Gamma(\mathbb{D} \textup{Av}^{G(\O)/I}_! j_{\tilde{w},!}) = \Gamma(\textup{Av}^{G(\O)/I}_* j_{\tilde{w},*}).$$

\begin{Lemma} \label{dualVermaWeyl}
	$\textup{Av}_*^{G(\O)/I} \mathbb{M}^{\kappa',\vee}_{-\lambda}[\dim G/B] \cong \mathbb{V}^{\kappa', \vee}_{-w_0(\lambda)-2\rho}$ for any dominant integral weight $\lambda$.
\end{Lemma}
\begin{proof}
	By the discussion preceeding this lemma, it suffices to prove the dual version of this isomorphism, namely $$\textup{Av}_!^{G(\O)/I} \mathbb{M}^{\kappa'}_{-\lambda}[-\dim G/B] \cong \mathbb{V}^{\kappa'}_{-w_0(\lambda)-2\rho}.$$
	The same argument as in the proof of Lemma \ref{VermaWeyl} reduces this to proving the $\g$-module isomorphism $$\textup{Av}^{G/B}_! M_{-\lambda}[-\dim G/B] \cong V_{-w_0(\lambda)-2\rho}.$$
	
	Since $\lambda$ is assumed dominant and integral, there is a unique dominant integral weight $\mu$ such that $w_0(\mu + \rho) = -\lambda + \rho$. Let $B1B$ be the identity $B$-orbit on $G/B$. We consider the $
	-w_0(\mu)$-twisted D-module $\textup{Dist}(B1B) \otimes_{\O_{G/B}} \mathcal{L}_{-w_0(\mu)}$, which corresponds to $M^{\vee}_{-\lambda}$ by the (twisted) Beilinson-Berstein localization. Then we have $\textup{Av}^{G/B}_*(\textup{Dist}(B1B) \otimes_{\O_{G/B}} \mathcal{L}_{-w_0(\mu)}) \cong \mathcal{L}_{-w_0(\mu)}[-\dim G/B]$. Taking global sections we obtain $$\textup{Av}^{G/B}_* M^{\vee}_{-\lambda}[\dim G/B] \cong V_{\mu} \equiv V_{-w_0(\lambda)-2\rho}.$$
	The assertion follows.
\end{proof}

\subsection{Generalized semi-infinite cohomology functor} \label{sec:gen_semiinf_functor}
Let $\mathsf{C}$ be a category acted on by $G(\K)$.
We define a functor $\mathsf{p}^- : \mathsf{C}^{N^-(\K)T(\O)} \to (\mathsf{C}^{T(\O)})_{N(\K)}$ as the composition $\textup{Av}^I_* : \mathsf{C}^{N^-(\K)T(\O)} \to \mathsf{C}^I$ followed by the equivalence $\mathsf{q}: \mathsf{C}^I \simeq (\mathsf{C}^{T(\O)})_{N(\K)}$.

\begin{Lemma} \label{lemma_p}
	The functor $\mathsf{p}^-$ is equivalent to $\textup{obliv}: \mathsf{C}^{N^-(\K)T(\O)} \to \mathsf{C}^{T(\O)}$ post-composed by the projection functor $\textup{proj}:\mathsf{C}^{T(\O)} \to (\mathsf{C}^{T(\O)})_{N(\K)}.$
\end{Lemma}
\begin{proof}
	This is essentially Section 2.1.3 in \cite{quantumsemiinf}.
\end{proof}

Assume that $\mathsf{C}$ is dualizable. Let $\langle\langle - , - \rangle\rangle : (\mathsf{C}^{\vee})^{N(\K)T(\O)} \times (\mathsf{C}^{T(\O)})_{N(\K)} \to \textup{Vect}$ be the natural pairing. Note that the natural pairing is characterized by 
\begin{equation} \label{natural_pairing}
\langle\langle \textup{Av}_!^{N(\K)} c^{\vee}, \mathsf{q}(c) \rangle\rangle \cong \langle c^{\vee}, c \rangle_I,
\end{equation}
for $c^{\vee} \in (\mathsf{C}^{\vee})^I$ and $c \in \mathsf{C}^I$.
In the followings, we take $\mathsf{C}=\hat{\g}_{\kappa} \textup{-mod}$ and its dual category $\mathsf{C}^{\vee}=\hat{\g}_{\kappa'} \textup{-mod}$ as in Section \ref{sec:affine_lie_algebra}. 

We define generalized semi-infinite cohomology functors (at positive level $\kappa$) by the following procedure: An object $\mathcal{F}$ in $\textup{D-mod}(\Gr_G)^{N^-(\K)T(\O)}$ defines a functor
\begin{equation*}
\mathsf{p}^-(\mathcal{F} \star_{G(\O)} -): \hat{\g}_{\kappa}\textup{-mod}^{G(\O)} \to (\hat{\g}_{\kappa}\textup{-mod}^{T(\O)})_{N(\K)}.
\end{equation*} 
We pair the resulting object with $\textup{Av}^{N(\K)}_! \mathbb{W}^{\kappa', *}_{-\mu-2\rho}[\dim(G/B)]$ to get the generalized semi-infinite cohomology functor corresponding to $\mathcal{F}$ and weight $\mu$:
\begin{equation} \label{def_gen_functor}
\semiinf_{\mathcal{F}}(\n(\K),-)^{\mu}:=\langle\langle \textup{Av}^{N(\K)}_! \mathbb{W}^{\kappa', *}_{-\mu-2\rho}[\dim(G/B)] \, ,\,\mathsf{p}^-(\mathcal{F} \star_{G(\O)} -)\rangle\rangle: \hat{\g}_{\kappa}\textup{-mod}^{G(\O)} \to \textup{Vect}.
\end{equation} 
To justify that it really is a generalization of the usual semi-infinite cohomology, we let $\mathcal{F} = \delta_{G(\O)}$ be the identity with respect to $\star_{G(\O)}$. Indeed, by \eqref{natural_pairing} and Proposition \ref{wakimoto_pairing}
\begin{equation*}
\begin{split}
\semiinf_{\delta_{G(\O)}}(&\n(\K),M)^\mu \cong \langle\langle \textup{Av}^{N(\K)}_! \mathbb{W}^{\kappa', *}_{-\mu-2\rho}[\dim(G/B)] \, ,\,\mathsf{p}^-(M)\rangle\rangle \cong \\
& \cong \langle \mathbb{W}^{\kappa', *}_{-\mu-2\rho}[\dim(G/B)] \, ,\, \textup{Av}^I_* (M) \rangle_I \cong \semiinf(\n(\K), M)^{\mu}.
\end{split}
\end{equation*}

The most important generalized semi-infinite cohomology functor for us will be the !*-semi-infinite cohomology functor, given by the semi-infinite IC object $\textup{IC}^{\frac{\infty}{2},-}$ in $\textup{D-mod}(\Gr_G)^{N^-(\K)T(\O)}$ which we construct below.

Consider $\check{\Lambda}^+$ as a poset with the (non-standard) partial order $$\check{\lambda}_1 \leq \check{\lambda}_2 \Leftrightarrow \check{\lambda}_2 - \check{\lambda}_1 \in \check{\Lambda}^+.$$ 
Recall the semi-infinite IC object $$\textup{IC}^{\frac{\infty}{2}, -}:= \underset{\check{\lambda} \in \check{\Lambda}^+}{\textup{colim}}\, t^{\check{\lambda}} \cdot \textup{Sat}((V_{\check{\lambda}})^*)[\langle 2\rho,\check{\lambda} \rangle] \in \textup{D-mod}(\Gr_G)^{N^-(\K)T(\O)},$$
analogous to $\textup{IC}^{\frac{\infty}{2}} \in \textup{D-mod}(\Gr_G)^{N(\K)T(\O)}$ constructed and studied in \cite{semiinf}.
For $\check{\mu} \in \check{\Lambda}^+$, the transition morphsim $$t^{\check{\lambda}} \cdot \textup{Sat}((V_{\check{\lambda}})^*)[\langle 2\rho,\check{\lambda} \rangle] \to t^{\check{\lambda}+\check{\mu}} \cdot \textup{Sat}((V_{\check{\lambda}+\check{\mu}})^*)[\langle 2\rho,\check{\lambda}+\check{\mu} \rangle]$$ in the colimit is given by
\begin{equation} \label{transition map}
\begin{split}
&t^{\check{\lambda}} \cdot \textup{Sat}((V_{\check{\lambda}})^*)[\langle 2\rho,\check{\lambda} \rangle] \to t^{\check{\lambda}} \cdot (t^{\check{\mu}} \cdot \textup{Sat}((V_{\check{\mu}})^*)[\langle 2\rho,\check{\mu} \rangle]) \star \textup{Sat}((V_{\check{\lambda}})^*)[\langle 2\rho,\check{\lambda} \rangle] \\
&\to t^{\check{\lambda}+\check{\mu}} \cdot \textup{Sat}((V_{\check{\lambda}})^* \otimes (V_{\check{\mu}})^*)[\langle 2\rho,\check{\lambda}+\check{\mu} \rangle] 
\to t^{\check{\lambda}+\check{\mu}} \cdot \textup{Sat}((V_{\check{\lambda}+\check{\mu}})^*)[\langle 2\rho,\check{\lambda}+\check{\mu} \rangle].
\end{split}
\end{equation}
Here the first arrow is given by the natural map $\delta_{t^{-\check{\mu}}G(\O)} \to  \textup{Sat}((V_{\check{\mu}})^*)[\langle 2\rho,\check{\mu} \rangle]$, which in turn is induced from the identification of the !-fiber of $\textup{Sat}((V_{\check{\mu}})^*)$ at the coset $t^{-\check{\mu}}G(\O)$ with $\C$ (see \eqref{!stalk_IC}).
The second arrow follows from the geometric Satake equivalence, and the third arrow arises from the dual of the embedding $V_{\check{\lambda}+\check{\mu}} \to V_{\check{\lambda}} \otimes V_{\check{\mu}}$.

If $\check{\lambda}$ is dominant, it is well-known that $\dim \overline{It^{\check{\lambda}}I} = \langle 2\rho,\check{\lambda} \rangle$. Then by Lemma \ref{averaging_id} we have $\textup{Av}^I_*(t^{\check{\lambda}} \cdot F) \cong j_{\check{\lambda},*} \star F [-\langle 2\rho, \check{\lambda} \rangle]$ for any $G(\O)$-equivariant $F$. Applying the functor $\textup{Av}^I_*$ to $\textup{IC}^{\frac{\infty}{2},-}$, we see that the transition morphism \eqref{transition map} becomes
\begin{equation*}
\begin{split}
&j_{\check{\lambda},*} \star \textup{Sat}((V_{\check{\lambda}})^*) \to j_{\check{\lambda},*} \star (j_{\check{\mu},*} \star \textup{Sat}((V_{\check{\mu}})^*)) \star \textup{Sat}((V_{\check{\lambda}})^*) \\
&\to j_{\check{\lambda}+\check{\mu},*} \star \textup{Sat}((V_{\check{\lambda}})^* \otimes (V_{\check{\mu}})^*) 
\to j_{\check{\lambda}+\check{\mu},*} \star \textup{Sat}((V_{\check{\lambda}+\check{\mu}})^*),
\end{split}
\end{equation*}
where the first arrow comes from $\delta_{G(\O)} \to j_{\check{\mu},*} \star j_{-\check{\mu},!} \star \delta_{G(\O)} \to j_{\check{\mu},*} \star \textup{Sat}((V_{\check{\mu}})^*)$, induced by the natural morphism $\textup{Sat}(V_{\check{\mu}}) \cong \textup{IC}_{\Gr_G^{\check{\mu}}} \to j_{\check{\mu},*} \star \delta_{G(\O)}$. We conclude that 
\begin{equation} \label{ave_IC}
\textup{Av}^{I}_* ( \textup{IC}^{\frac{\infty}{2},-} \star_{G(\O)} M) \cong \underset{\check{\lambda} \in \check{\Lambda}^+}{\textup{colim}}\, j_{\check{\lambda}, *} \star_I \textup{Sat}((V_{\check{\lambda}})^*) \star_{G(\O)} M
\end{equation}
in the category $\textup{D-mod(Gr)}^{I}$.

Now we consider the generalized semi-infinite cohomology functor $\semiinf_{\textup{IC}^{\frac{\infty}{2},-}}(\n(\K),-)^\mu$. By definition
\begin{equation*}
\begin{split}
\semiinf_{\textup{IC}^{\frac{\infty}{2},-}}(\n(\K),M)^\mu = &\langle\langle \textup{Av}^{N(\K)}_! \mathbb{W}^{\kappa', *}_{-\mu-2\rho}[\dim(G/B)] \, ,\, \mathsf{p}^-(\textup{IC}^{\frac{\infty}{2},-} \star_{G(\O)} M) \rangle\rangle \\ 
&= \langle\langle \textup{Av}^{N(\K)}_! \mathbb{W}^{\kappa', *}_{-\mu-2\rho}[\dim(G/B)] \, ,\, \mathsf{q} \circ \textup{Av}^I_* (\textup{IC}^{\frac{\infty}{2},-} \star_{G(\O)} M) \rangle\rangle,
\end{split}
\end{equation*}
which again by \eqref{natural_pairing}, \eqref{ave_IC} and Proposition \ref{wakimoto_pairing} is isomorphic to
\begin{equation} \label{semiinf_as_pairing}
\begin{split}
&\langle \mathbb{W}^{\kappa', *}_{-\mu-2\rho}[\dim(G/B)] \, ,\, \textup{Av}^I_* (\textup{IC}^{\frac{\infty}{2},-} \star_{G(\O)} M) \rangle_I \\ \cong&\,\, \semiinf(\n(\K), \textup{Av}^I_*(\textup{IC}^{\frac{\infty}{2},-}) \star_{G(\O)} M)^{\mu} \cong \underset{\check{\lambda} \in \check{\Lambda}^+}{\textup{colim}} \, \semiinf(\n(\K) ,\, j_{\check{\lambda}, *} \star_I \textup{Sat}((V_{\check{\lambda}})^*) \star_{G(\O)} M )^\mu.
\end{split}
\end{equation}

\begin{Definition} \label{gensemiinf}
	The functor $\semiinf_{!*}(\n(\K), -)^{\mu}: \hat{\g}_{\kappa} \textup{-mod}^{G(\O)} \to \textup{Vect}$ defined by
	\begin{equation*}
	M \mapsto \underset{\check{\lambda} \in \check{\Lambda}^+}{\textup{colim}} \, \semiinf(\n(\K) ,\, j_{\check{\lambda}, *} \star_I \textup{Sat}((V_{\check{\lambda}})^*) \star_{G(\O)} M )^\mu
	\end{equation*}
	is called the $\mu$-component of the !*-generalized semi-infinite cohomology functor.
\end{Definition}

We similarly define the $\mu$-component of the !-generalized semi-infinite cohomology functor as
\begin{equation*}
\semiinf_!(\n(\K),-)^\mu:= \underset{\check{\lambda} \in \check{\Lambda}^+}{\textup{colim}} \, \semiinf(\n(\K) ,\, j_{\check{\lambda}, *} \star_I j_{-\check{\lambda}, *} \star_{G(\O)} -)^\mu
\end{equation*}
The original semi-infinite cohomology functor should then be regarded as the *-version of the construction.

\subsection{The formulas} \label{sec:formulas}
Recall that we have fixed a positive level $\kappa$ and the quantum parameter is set to $q=\textup{exp}(\frac{\pi \sqrt{-1}}{\kappa' - \kappa_{\textup{crit}}})$. Let $A$ be one of the algebras $U^{\textup{KD}}_q(\n)$, $\u_q(\n)$, or $U^{\textup{Lus}}_q(\n)$ defined in Section \ref{sec:quantum_groups}.  Set $\textup{C}^{\bullet}(A, -)$ to be the derived functor of $A$-invariants, and $\textup{C}^{\bullet}(A, -)^{\mu}$ the $\mu$-component of the resulting (complex of) $\Lambda$-graded vector spaces. Here, we take as input a $U^{\textup{Lus}}_q(\g)$-module, regarded as an $A$-module via restriction.

The goal of this chapter is to prove the following formula:
\begin{Theorem} \label{*thm}
	For each weight $\mu$ we have an isomorphism in $\textup{Vect}$:
	\begin{equation*}
	\semiinf(\n(\K), M)^{\mu} \cong \textup{C}^{\bullet}(U^{\textup{Lus}}_q(\n), \textup{KL}^{\kappa}_G (M))^{\mu}.
	\end{equation*}
\end{Theorem}

Our proof of Theorem \ref{*thm} relies on the following analogous formula for the !*-generalized semi-infinite cohomology at positive level:

\begin{Theorem} \label{!*thm}
	The isomorphism
	$$\semiinf_{!*}(\n(\K), M)^{\mu} \cong \textup{C}^{\bullet}(\u_q(\n), \textup{KL}^{\kappa}_G (M))^{\mu}$$ holds for all weights $\mu$.
\end{Theorem}

The proofs of Theorem \ref{*thm} and Theorem \ref{!*thm} are given in the next two sections.

We state the corresponding conjectural formula for the !-generalized semi-infinite cohomology functor:

\begin{Conjecture}
	For all $\mu \in \Lambda$,
	\begin{equation*}
	\semiinf_!(\n(\K), M)^{\mu} \cong \textup{C}^{\bullet}(U^{\textup{KD}}_q(\n), \textup{KL}^{\kappa}_G (M))^{\mu}.
	\end{equation*}
\end{Conjecture}

\subsection{Proof of Theorem \ref{!*thm}}
By the definition of the !*-functor
\begin{equation*}
\semiinf_{!*}(\n(\K), M)^{\mu} \cong \underset{\check{\lambda} \in \check{\Lambda}^+}{\textup{colim}} \, \langle \mathbb{W}^{\kappa', *}_{-\mu-2\rho}[\dim(G/B)]\, ,\,\, j_{\check{\lambda}, *} \star_I \textup{Sat}((V_{\check{\lambda}})^*) \star_{G(\O)} M \rangle_I.
\end{equation*}
We have by Lemma \ref{pairingconv}
\begin{equation*}
\begin{split}
\langle \mathbb{W}^{\kappa', *}_{-\mu-2\rho}&[\dim(G/B)]\, ,\, j_{\check{\lambda}, *} \star_I \textup{Sat}((V_{\check{\lambda}})^*) \star_{G(\O)} M \rangle_I \\ &\cong \langle j_{-\check{\lambda}, *} \star_I \mathbb{W}^{\kappa', *}_{-\mu-2\rho}[\dim(G/B)]\, ,\, \textup{Sat}((V_{\check{\lambda}})^*) \star_{G(\O)} M \rangle_I,
\end{split}
\end{equation*}
and by \eqref{*Wakimoto} the latter is isomorphic to
$$\langle \mathbb{W}^{\kappa', *}_{-\check{\lambda}-\mu-2\rho}[\dim(G/B)]\, ,\, \textup{Sat}((V_{\check{\lambda}})^*) \star_{G(\O)} M \rangle_I.$$
When $\check{\lambda}$ is sufficiently dominant, the above pairing becomes
$$\langle \mathbb{M}^{\kappa',\vee}_{-\check{\lambda}-\mu-2\rho}[\dim(G/B)]\, ,\, \textup{Sat}((V_{\check{\lambda}})^*) \star_{G(\O)} M \rangle_I,
$$
which is then isomorphic to 
\begin{equation*}
\Hom_{\hat{\g}_{\kappa}\textup{-mod}^I}\left(\mathbb{D}_I\,( \mathbb{M}^{\kappa',\vee}_{-\check{\lambda}-\mu -2\rho}[\dim(G/B)])\, ,\,  \textup{Sat}((V_{\check{\lambda}})^*) \star_{G(\O)} M \right)
\end{equation*}
by \eqref{pairing}.

Since $\textup{Sat}((V_{\check{\lambda}})^*) \star_{G(\O)} M$ is $G(\O)$-equivariant, by the left adjointness of $\textup{Av}^{G(\O)/I}_!$ and \eqref{duality_average}, the above is isomorphic to
\begin{equation*}
\Hom_{\hat{\g}_{\kappa}\textup{-mod}^{G(\O)}}\left(\mathbb{D}_{G(\O)}\, \textup{Av}^{G(\O)/I}_*( \mathbb{M}^{\kappa',\vee}_{-\check{\lambda}-\mu -2\rho}[\dim(G/B)])\, ,\,  \textup{Sat}((V_{\check{\lambda}})^*) \star_{G(\O)} M \right),
\end{equation*}
and from Lemma \ref{dualVermaWeyl} this becomes
\begin{equation} \label{hom_in_positive_level}
\Hom_{\hat{\g}_{\kappa}\textup{-mod}^{G(\O)}}\left(\mathbb{D}_{G(\O)}\,\mathbb{V}^{\kappa',\vee}_{-w_0(\check{\lambda}+\mu)}\, ,\,  \textup{Sat}((V_{\check{\lambda}})^*) \star_{G(\O)} M \right).
\end{equation}

Recall the dual quantum Weyl module $\VV^{\vee}_{\nu}$, defined as the image of $\mathbb{V}^{\kappa',\vee}_{\nu}$ under the negative level Kazhdan-Lusztig functor; i.e. $\VV^{\vee}_{\nu} := \textup{KL}_G(\mathbb{V}^{\kappa',\vee}_{\nu})$. Then with $\mathbb{D}^q (\VV^{\vee}_{\nu}) \cong \VV_{-w_0(\nu)}$ and the definition of $\textup{KL}^{\kappa}_G$, we deduce that \eqref{hom_in_positive_level} is isomorphic to
\begin{equation*}
\Hom_{U^{\textup{Lus}}_q(\g)}\left(\VV_{\mu+\check{\lambda}}, \textup{KL}^{\kappa}_G(\textup{Sat}((V_{\check{\lambda}})^*) \star_{G(\O)} M)\right).
\end{equation*}
Since $\textup{Sat}(V_{\check{\nu}}) \cong \textup{IC}_{\check{\nu}}$ is Verdier self-dual, $\mathbb{D}_{G(\O)}(\textup{Sat}(V_{\check{\nu}}) \star M) \cong \textup{Sat}(V_{\check{\nu}}) \star \mathbb{D}_{G(\O)}(M)$ by the same argument as in the proof of \eqref{dualconv}. Combined with \eqref{convolution_Fr} we have
\begin{equation*}
\begin{split}
&\textup{KL}^{\kappa}_G(\,\textup{Sat}((V_{\check{\lambda}})^*) \star_{G(\O)} M\,) \cong \mathbb{D}^q \circ \textup{KL}_G (\,\textup{Sat}((V_{\check{\lambda}})^*) \star_{G(\O)} \mathbb{D}_{G(\O)}M\,) \cong \\
\cong &\, \mathbb{D}^q( \textup{Fr}_q((V_{\check{\lambda}})^*))\otimes \textup{KL}^{\kappa}_G(M) \cong \textup{Fr}_q(V_{\check{\lambda}})\otimes \textup{KL}^{\kappa}_G(M).
\end{split}
\end{equation*}

So far we have shown that
\begin{equation} \label{formula_semiinf_side}
\semiinf_{!*}(\n(\K), M)^{\mu} \cong \underset{\check{\lambda} \in \check{\Lambda}^+}{\textup{colim}}\,\Hom_{U^{\textup{Lus}}_q(\g)}(\VV_{\mu+\check{\lambda}}, \textup{Fr}_q(V_{\check{\lambda}})\otimes \textup{KL}^{\kappa}_G(M)).
\end{equation} 

On the quantum group side, we follow the same derivation as in the proof of \cite[Theorem 3.2.2]{quantumsemiinf}. Recall $\overset{\bullet}{\u}_q(\b)\textup{-mod}$ the category of representations of the small quantum Borel with full Lusztig's torus. The coinduction functor $\textup{CoInd}^{U^{\textup{Lus}}_q(\b)}_{\overset{\bullet}{\u}_q(\b)}$ is the right adjoint to the restriction functor from $U^{\textup{Lus}}_q(\b)\textup{-mod}$ to $\overset{\bullet}{\u}_q(\b)\textup{-mod}$. The functor $\textup{CoInd}^{U^{\textup{Lus}}_q(\b)}_{\overset{\bullet}{\u}_q(\b)}$ sends the trivial representation to
$$\textup{CoInd}^{U^{\textup{Lus}}_q(\b)}_{\overset{\bullet}{\u}_q(\b)}(\C) \cong \textup{Fr}_q(\O_{\check{B}/\check{T}})
$$
by \cite[Section 3.1.4]{quantumsemiinf}. Moreover by \cite[Proposition 3.1.2]{ABBGM} we have $\O_{\check{B}/\check{T}} \simeq \underset{\check{\lambda}\in \check{\Lambda}^+}{\textup{colim}}\, \C_{-\check{\lambda}}\otimes V_{\check{\lambda}}$ as $\check{B}$-modules. Then
\begin{equation*}
\begin{split}
& \textup{C}^{\bullet}(\u_q(\n), \textup{KL}^{\kappa}_G (M))^{\mu} := \Hom_{\overset{\bullet}{\u}_q(\b)}(\C_\mu, \, \textup{Res}^{U^{\textup{Lus}}_q(\g)}_{\overset{\bullet}{\u}_q(\b)}\textup{KL}^{\kappa}_G (M)) \\
\cong\, & \Hom_{U^{\textup{Lus}}_q(\b)}(\C_\mu, \, \textup{CoInd}^{U^{\textup{Lus}}_q(\b)}_{\overset{\bullet}{\u}_q(\b)} \circ \textup{Res}^{U^{\textup{Lus}}_q(\g)}_{\overset{\bullet}{\u}_q(\b)}\textup{KL}^{\kappa}_G (M)) \\
\cong\, & \Hom_{U^{\textup{Lus}}_q(\b)}(\C_\mu, \, \textup{Fr}_q(\O_{\check{B}/\check{T}})\otimes \textup{Res}^{U^{\textup{Lus}}_q(\g)}_{\overset{\bullet}{\u}_q(\b)}\textup{KL}^{\kappa}_G (M)) \\
\cong\, & \underset{\check{\lambda} \in \check{\Lambda}^+}{\textup{colim}}\,\Hom_{U^{\textup{Lus}}_q(\b)}(\C_\mu, \, \textup{Fr}_q(\C_{-\check{\lambda}}\otimes V_{\check{\lambda}})\otimes \textup{Res}^{U^{\textup{Lus}}_q(\g)}_{\overset{\bullet}{\u}_q(\b)}\textup{KL}^{\kappa}_G (M)) \\
\cong\, &  \underset{\check{\lambda} \in \check{\Lambda}^+}{\textup{colim}}\,\Hom_{U^{\textup{Lus}}_q(\b)}(\C_{\mu +\check{\lambda}}, \, \textup{Fr}_q( V_{\check{\lambda}})\otimes \textup{Res}^{U^{\textup{Lus}}_q(\g)}_{\overset{\bullet}{\u}_q(\b)}\textup{KL}^{\kappa}_G (M))\\
\cong\, &  \underset{\check{\lambda} \in \check{\Lambda}^+}{\textup{colim}}\,\Hom_{U^{\textup{Lus}}_q(\g)}(\VV_{\mu+\check{\lambda}}, \textup{Fr}_q(V_{\check{\lambda}})\otimes \textup{KL}^{\kappa}_G(M)),
\end{split}
\end{equation*}
which agrees with \eqref{formula_semiinf_side}.

\subsection{Proof of Theorem \ref{*thm}} \label{sec:proof_*thm}
The proof follows the same idea as in \cite[Section 3.3]{quantumsemiinf} for the negative level case.

Let $\overset{\bullet}{\mathcal{F}_{!*}}$ denote the object $$\underset{\check{\nu} \in \check{\Lambda}}{\bigoplus}\,\left( \underset{\check{\lambda} \in \check{\Lambda}^+}{\textup{colim}}\, j_{\check{\nu}+\check{\lambda}, *} \star \textup{Sat}((V_{\check{\lambda}})^*)\right)$$ in the category $\textup{D-mod}(\Gr_G)^I$.

By the theory of Arkhipov-Bezrukavnikov-Ginzburg \cite{ABG}, we have an equivalence $$\textup{D-mod}(\Gr_G)^I \simeq \textup{IndCoh}((\textup{pt} \times_{\check{\g}} \tilde{\check{\mathcal{N}}}) / \check{G}).$$ Under this equivalence, $\overset{\bullet}{\mathcal{F}_{!*}}$ corresponds to $s_* \O(\check{B})$, where
$$s: \textup{pt} / \check{B} \simeq (\check{G}/\check{B})/\check{G} \hookrightarrow (\textup{pt} \times_{\check{\g}} \tilde{\check{\mathcal{N}}}) / \check{G}.
$$
It follows that $\overset{\bullet}{\mathcal{F}_{!*}}$ is equipped with a $\check{B}$-action, such that 
the $\check{B}$-invariant of $\overset{\bullet}{\mathcal{F}_{!*}}$ corresponds to $s_* \O(0)$, where $\O(0)$ is the trivial line bundle on $\textup{pt} / \check{B}$.

Again by the equivalence in \cite{ABG}, 
the delta function $\delta_{G(\O)}$ on the identity coset in $\Gr_G$ corresponds to $s_* \O(0)$. Consequently, 
the $\check{B}$-invariant of $\overset{\bullet}{\mathcal{F}_{!*}}$ is identified with $\delta_{G(\O)}$.

Now, we consider the object $$ \langle \mathbb{W}^{\kappa', *}_{-\mu-2\rho}[\dim G/B] \, ,\, \overset{\bullet}{\mathcal{F}_{!*}} \star_{G(\O)} M \rangle_I
$$ in $\textup{Vect}$. From the above discussion, this object inherits a $\check{B}$-action, such that the $\check{B}$-invariant is equal to $\langle \mathbb{W}^{\kappa', *}_{-\mu-2\rho}[\dim G/B] \, ,\, M \rangle_I \cong \semiinf(\n(\K), M)^{\mu}.$
By Theorem \ref{!*thm}, we have
\begin{equation*}
\begin{split}
\langle \mathbb{W}^{\kappa', *}_{-\mu-2\rho}[\dim G/B]& \, ,\, \overset{\bullet}{\mathcal{F}_{!*}} \star_{G(\O)} M \rangle_I \cong \underset{\check{\nu} \in \check{\Lambda}}{\bigoplus}\, \semiinf_{!*}(\n(\K), M)^{\mu+\check{\nu}}\\ &\cong \underset{\check{\nu} \in \check{\Lambda}}{\bigoplus}\,\textup{C}^{\bullet}(\u_q(\n), \textup{KL}^{\kappa}_G (M))^{\mu+\check{\nu}}.
\end{split}
\end{equation*}
It follows that $\semiinf(\n(\K), M)^{\mu}$ is isomorphic to the $\check{B}$-invariant of the space $$\underset{\check{\nu} \in \check{\Lambda}}{\bigoplus}\,\textup{C}^{\bullet}(\u_q(\n), \textup{KL}^{\kappa}_G (M))^{\mu+\check{\nu}} \cong \Hom_{\u_q(\b)}(\C_{\mu}, \textup{KL}^{\kappa}_G (M)),$$ which is identified with $$\Hom_{U^{\textup{Lus}}_q(\b)}(\C_{\mu}, \textup{KL}^{\kappa}_G (M)) \cong \textup{C}^{\bullet}(U^{\textup{Lus}}_q(\n), \textup{KL}^{\kappa}_G (M))^{\mu}.
$$
This proves the theorem.

\section{An algebraic approach at irrational level} \label{chapter:irr_lvl}
The theory we developed so far is greatly simplified if we assume the level is irrational. As the quantum parameter $q$ is no longer a root of unity, Lusztig's, Kac-De Concini's and the small quantum groups all coincide, with the category of representations being semi-simple. The category of $G(\O)$-equivariant representations of the affine Lie algebra is semi-simple at irrational level as well, and in this case the (negative level) Kazhdan-Lusztig functor is an equivalence tautologically. That Kazhdan-Lusztig functor is a monoidal functor can be seen as a reformulation of Drinfeld's theorem on Knizhnik-Zamolodchikov associators \cite[Part III and Part IV]{KazLus}.

Fix an irrational level $\kappa$ throughout this chapter. We will give an algebraic proof of the formula that appears in Theorem \ref{!*thm} at an irrational level. 

\subsection{BGG-type resolutions}
Let $\ell$ be the usual length function on the Weyl group $W$ of $\g$, and recall the dot action of the Weyl group on weights by $w \cdot \mu := w(\mu + \rho)-\rho$. Then the celebrated \textit{Bernstein-Gelfand-Gelfand} (BGG) \textit{resolution} is stated as follows:

\begin{Proposition} \label{th:BGG}
	Let $\lambda$ be a dominant integral weight of $\g$. Then we have a resolution of $V_{\lambda}$ given by
	$$0 \to M_{w_0 \cdot \lambda} \to \dotsb \to \bigoplus_{\ell(w)=i} M_{w \cdot \lambda} \to \dotsb \to M_{\lambda} \twoheadrightarrow V_{\lambda}.$$
\end{Proposition}

Note that $\mathbb{V}^{\kappa}_{\lambda} \cong (\mathbb{V}^{\kappa}_{\lambda})^{\vee}$ since it is irreducible (when $\kappa$ is irrational). We apply the induction functor $(\cdot)^{\kappa}$ to the BGG resolution and then take the contragredient dual to get
\begin{equation} \label{eq:dualBGG}
\mathbb{V}^{\kappa}_{\lambda} \hookrightarrow (\mathbb{M}^{\kappa}_{\lambda})^{\vee} \to \dotsb \to \bigoplus_{\ell(w)=i}(\mathbb{M}^{\kappa}_{w \cdot \lambda})^{\vee} \to \dotsb \to (\mathbb{M}^{\kappa}_{w_0 \cdot \lambda})^{\vee} \to 0.
\end{equation}
Combine Theorem \ref{th:WakiVer} and \eqref{eq:dualBGG} we get 
\begin{equation*}
\mathbb{V}^{\kappa}_{\lambda} \hookrightarrow \mathbb{W}^{\kappa,w_0}_{\lambda} \to \dotsb \to \bigoplus_{\ell(w)=i}\mathbb{W}^{\kappa,w_0}_{w \cdot \lambda} \to \dotsb \to \mathbb{W}^{\kappa,w_0}_{w_0 \cdot \lambda} \to 0.
\end{equation*}

Now we can compute the semi-infinite cohomology of Weyl modules at irrational level by applying Proposition \ref{th:WakiBRST}:

\begin{Corollary} \label{MainThm}
	For $\lambda \in \Lambda^+$, we have an isomorphism of $\hat{\t}_{\kappa + \textup{shift}}$-modules
	$$H^{\frac{\infty}{2}+i}(\n(\K), \mathbb{V}^{\kappa}_{\lambda}) \cong \bigoplus_{\ell(w)=i} \pi^{\kappa + \textup{shift}}_{w \cdot \lambda}.$$
\end{Corollary}

Now we turn to the quantum group side. When $q$ is not a root of unity, the quantum Weyl module $\mathcal{V}_{\lambda}$ coincides with the irreducible module $\LL_{\lambda}$, constructed by the usual procedure of taking irreducible quotient of the quantum Verma module $\MM_{\lambda}$.
Analogous to the non-quantum case, we have the BGG resolution for representations of $U_q(\g)$:
$$0 \to \MM_{w_0 \cdot \lambda} \to \dotsb \to \bigoplus_{\ell(w)=i} \MM_{w \cdot \lambda} \to \dotsb \to \MM_{\lambda} \twoheadrightarrow \mathcal{V}_{\lambda}.$$
As a consequence, we deduce 
\begin{equation} \label{eq:quantum_inv}
H^i({U_q(\n)}, \mathcal{V}_{\lambda}) = \bigoplus_{\ell(w)=i} \Hom_{U_q(\n)}(\C, \MM_{w \cdot \lambda}) = \bigoplus_{\ell(w)=i} \C_{w \cdot \lambda}.
\end{equation}

\subsection{Commutativity of the diagram} \label{sec:comm_diag}
Recall the tautological equivalence $\textup{KL}_T: \hat{\t}_{\kappa \textup{+shift}} \textup{-mod}^{T(\O)} \to \textup{Rep}_q(T)$ which is induced by the assignment
$$\pi^{\kappa \textup{+shift}}_{\lambda} \mapsto \C_{\lambda}.$$
We will verify the commutativity of the following diagram
\begin{equation} \label{diagram_irr}
\xymatrix @C=6pc{
	\hat{\g}_{\kappa} \textup{-mod}^{G(\O)} \ar[d]_{\textup{KL}_G} \ar[r]^{\semiinf(\n (\K), -)} & \hat{\t}_{\kappa \textup{+shift}} \textup{-mod}^{T(\O)} \ar[d]^{\textup{KL}_T} \\
	U_q(\g)\textup{-mod} \ar[r]^{\textup{C}^\bullet(U_q(\n), -)} & \textup{Rep}_q(T)}
\end{equation}
which clearly implies
$$\semiinf(\n(\K), M)^{\mu} \cong \textup{C}^{\bullet}(U_q(\n), \textup{KL}_G (M))^{\mu}
$$ for all $\mu$.

We will need the theory of compactly generated categories.
A detailed treatise of the theory is given in \cite{Lurie}. For a brief review of definitions and facts, see \cite{DrGa}.
We recall the following proposition from \cite{Lurie}.

\begin{Proposition} [\cite{Lurie} Proposition 5.3.5.10] \label{Lu}
	Let $\mathsf{C}$ be a cocomplete category, $\mathsf{D}$ be a small category, and $F: \mathsf{D} \to \mathsf{C}$ be a functor. Then $F$ uniquely induces a continuous functor $\bar{F}: \mathsf{Ind(D)} \to \mathsf{C}$ with $\bar{F}|_{\mathsf{D}} = F$. Here, $\mathsf{Ind(D)}$ denotes the ind-completion of the category $\mathsf{D}$.
\end{Proposition}

We shall take $\mathsf{D}$ to be the full subcategory of compact generators of $\hat{\g}_{\kappa} \textrm{-mod}^{G(\O)}$, i.e., the subcategory whose objects consist of Weyl modules $\mathbb{V}^{\kappa}_{\lambda}$ for dominant integral weights $\lambda$. Then from the definition of compactly generated categories, we have $\mathsf{Ind(D)} = \hat{\g}_{\kappa} \textrm{-mod}^{G(\O)}$. Let $\mathsf{C}$ be the category $\textup{Rep}_q(T)$.

To prove the commutativity of \eqref{diagram_irr}, by Proposition \ref{Lu} it suffices to show that the two functors $\textup{C}^\bullet(U_q(\n),-) \circ \textup{KL}_G$ and $\textup{KL}_T \circ \semiinf(\n(\K), -)$ restricted to $\mathsf{D}$ are the same.

Since $\textup{KL}_G (\mathbb{V}^{\kappa}_{\lambda}) = \mathcal{V}_{\lambda}$, from \eqref{eq:quantum_inv} and Corollary \ref{MainThm} we get
$$H^i(\textup{C}^\bullet(U_q(\n),-) \circ \textup{KL}_G (\mathbb{V}^{\kappa}_{\lambda})) \cong H^i(\textup{KL}_T \circ \semiinf(\n(\K), -)(\mathbb{V}^{\kappa}_{\lambda})) \quad \forall i.$$
As objects in the category $\textup{Rep}_q(T)^\heartsuit$ have no nontrivial extensions, this shows that the two functors $\textup{C}^\bullet(U_q(\n),-) \circ \textup{KL}_G$ and $\textup{KL}_T \circ \semiinf(\n(\K), -)$ have the same image for objects in $\mathsf{D}$.

Let $\lambda$ and $\mu$ be arbitrary dominant integral weights of $\g$, and let $f$ be a morphism in $\Hom_{\hat{\g}_{\kappa}}(\mathbb{V}^{\kappa}_{\lambda}, \mathbb{V}^{\kappa}_{\mu})$. It remains to show that the following two morphisms
$$\textup{C}^\bullet(U_q(\n),-) \circ \textup{KL}_G (f)$$
and
$$\textup{KL}_T \circ \semiinf(\n(\K), -)(f)$$
are identical.

Since $\kappa$ is irrational, by \cite[Proposition~27.4]{KazLus} all Weyl modules $\mathbb{V}^{\kappa}_{\lambda}$ are irreducible. Then $\Hom_{\hat{\g}_{\kappa}}(\mathbb{V}^{\kappa}_{\lambda}, \mathbb{V}^{\kappa}_{\mu})$ vanishes when $\lambda \neq \mu$. Now, $\Hom_{\hat{\g}_{\kappa}}(\mathbb{V}^{\kappa}_{\lambda}, \mathbb{V}^{\kappa}_{\lambda})$ is one-dimensional and generated by the identity morphism $\textup{Id}_{\mathbb{V}^{\kappa}_{\lambda}}$. Clearly we have
\begin{equation*}
\textup{C}^\bullet(U_q(\n),-) \circ \textup{KL}_G \,(\textup{Id}_{\mathbb{V}^{\kappa}_{\lambda}})  
=  \textup{Id}_{C^{\bullet}} 
=  \textup{KL}_T \circ \semiinf(\n(\K), -)(\textup{Id}_{\mathbb{V}^{\kappa}_{\lambda}}),
\end{equation*}
where $C^{\bullet}$ is the complex in $\textup{Rep}_q(T)$ with $C^i = \bigoplus_{\ell(w)=i} \C_{w \cdot \lambda}$.
It follows that $\textup{C}^\bullet(U_q(\n),-) \circ \textup{KL}_G$ and $\textup{KL}_T \circ \semiinf(\n(\K), -)$ agree on morphisms in $\mathsf{D}$. Therefore the two functors are isomorphic when restricted to $\mathsf{D}$ and so the diagram \eqref{diagram_irr} commutes.

\bibliographystyle{alphanum}
\bibliography{semiinf_positive_arxiv}

\end{document}